\documentclass[11pt]{amsart}
\usepackage{amssymb,amsxtra,amsmath,amsfonts}
\usepackage{verbatim}
\usepackage{graphicx,color,graphics} 
\newtheorem{theorem}{Theorem}
\newtheorem{lemma}[theorem]{Lemma}
\newtheorem{proposition}[theorem]{Proposition}

\theoremstyle{definition}
\newtheorem{definition}[theorem]{Definition}
\newtheorem{example}[theorem]{Example}

\theoremstyle{remark}
\newtheorem{remark}[theorem]{Remark}

\numberwithin{equation}{section}
\numberwithin{theorem}{section}

\newcommand\thref{Theorem \ref}
\newcommand\leref{Lemma \ref}
\newcommand\prref{Proposition \ref}

\newcommand\deref{Definition \ref}
\newcommand\exref{Example \ref}
\newcommand\reref{Remark \ref}
\newcommand\seref{Section \ref}
\renewcommand{\comment}[1]{}
\def\CC{\mathbb{C}}

\def\V{\mathcal{V}}
\def\W{\mathcal{W}}

\def\C{\mathbb{C}}

\def\D{\mathcal{D}}

\def\F{\mathcal{F}}

\def\N{\mathcal{N}}

\def\QQ{\mathbb{Q}}

\def\S{\mathcal{S}}

\def\ZZ{\mathbb{Z}}

\DeclareMathOperator\Aut{Aut}
\DeclareMathOperator\End{End}
\DeclareMathOperator\Der{Der}
\DeclareMathOperator\Hom{Hom}

\DeclareMathOperator\ad{ad}

\DeclareMathOperator\tr{tr}

\DeclareMathOperator\Res{Res}
\DeclareMathOperator\Ind{Ind}
\DeclareMathOperator\Span{span}
\DeclareMathOperator\QF{Fie}
\DeclareMathOperator\LF{LFie}
\DeclareMathOperator\PLF{LFie}
\def\vac{{\boldsymbol{1}}}  %{|0\rangle} % vacuum vector

\def\ii{\mathrm{i}} %{\sqrt{-1}}
\def\al{\alpha}
\def\be{\beta}
\def\ga{\gamma}

\def\de{\delta}
\def\De{\Delta}
\def\io{\iota}

\def\ph{\varphi}
\def\ka{\kappa}
\def\la{\lambda}
\def\La{\Lambda}
\def\si{\sigma}

\def\om{\omega}
\def\ze{\zeta}
\def\z{z}
\def\d{\partial}
\def\Re{\mathrm{Re}\,}
\def\Im{\mathrm{Im}\,}
\def\lieg{{\mathfrak{g}}}
\def\lieh{{\mathfrak{h}}}
\def\hhp{\hat\lieh_\ph}

\begin{document}

\title[Twisted logarithmic modules of vertex algebras]{Twisted logarithmic modules of vertex algebras}
\author{Bojko Bakalov}
\address{Department of Mathematics,
North Carolina State University,
Raleigh, NC 27695, USA}
\email{bojko\_bakalov@ncsu.edu}

%\thanks{The first author is supported in part by NSF and NSA grants} %DMS-0701011 

\date{April 24, 2015}



\begin{abstract}
Motivated by logarithmic conformal field theory and Gromov--Witten theory, we introduce a notion of a twisted module
of a vertex algebra under an arbitrary (not necessarily semisimple) automorphism. Its main feature is that the twisted fields involve the logarithm of the formal variable. We develop the theory of such twisted modules and, in particular, derive a Borcherds identity 
and commutator formula for them. We investigate in detail the examples of affine and Heisenberg vertex algebras.
\end{abstract}

\maketitle

%\tableofcontents

\section{Introduction}\label{s1}

The vertex operator realizations of affine Kac--Moody algebras \cite{LW,FK,KKLW,KP} led to the
introduction of the notions of a \emph{vertex algebra} \cite{B} and its \emph{twisted modules} \cite{Le,FLM,FFR,D}.
Twisted modules played an important role in the Frenkel--Lepowsky--Meurman construction of a 
vertex algebra with a natural action of the Monster on it \cite{FLM}.
Vertex algebras provide a rigorous algebraic description of two-dimensional 
chiral \emph{conformal field theory} (see e.g.\ \cite{BPZ, Go, DMS}), and twisted modules are important for 
studying orbifolds (see e.g.\ \cite{DHVW,DVVV, KT} among many other works).

Motivated by an example from \emph{logarithmic conformal field theory} (see e.g.\ \cite{AM,CR}),
Y.-Z.~Huang introduced in \cite{H} a more general notion of a twisted module, for which the corresponding automorphism may have an infinite order and is not necessarily semisimple. The main feature of such twisted modules
is that the twisted fields involve the logarithm of the formal variable. However, they lacked a Borcherds identity, $n$-th product identity, or commutator formula, all of which are powerful tools in the theory of vertex algebras. The difficulty was partly caused by the fact that the definition of $n$-th product of fields from \cite{Li1,Li2} is not very convenient in the case of twisted modules. This problem was solved in \cite{BM}, where we showed that another formula for the $n$-th product \cite{BN,BK2} remains valid in the twisted case.

In the present paper, we use the formula from \cite{BM} to provide another definition of a twisted module, more general than the one from \cite{H}. Our definition is in the spirit of \cite{Li1,Li2,LL}, so that the state-field correspondence map $Y$ is a homomorphism of vertex algebras relative to all $n$-th products. We develop a framework that allows many results about vertex algebras to be transferred to general twisted modules. In particular, we define a mode expansion of twisted fields and a shifted delta function.
Our main results are a Borcherds identity and a commutator formula for general twisted modules.
We investigate in detail the examples of affine and Heisenberg vertex algebras, and we plan to consider additional examples in the future.
The theory developed here will be used in our joint work with T.~Milanov, which aims to understand and utilize the vertex operators arising in \emph{Gromov--Witten theory} (see \cite{DZ1,DZ2,M,MT,FGM,BM,CV,LYZ,MST}).

Here is an outline of the present paper.
In \seref{s2}, we briefly review the basic definitions and properties of vertex algebras and their modules. This section can be skipped by readers familiar with the theory. 

In \seref{s3}, we introduce the notions of a logarithmic field, locality, and $n$-th products of logarithmic fields. We express the $n$-th product in terms of the normally ordered product and the propagator, and we prove that any local collection of logarithmic fields generates a vertex algebra. 

In \seref{s4}, we introduce the main object of the paper, the notion of a $\ph$-twisted $V$-module where $\ph$ is an arbitrary (not necessarily semisimple) automorphism of a vertex algebra $V$. When $\ph$ is locally finite, we express it as $\ph=\si e^{-2\pi\ii\N}$, where $\si\in\Aut(V)$ is semisimple and $\N\in\Der(V)$ is locally nilpotent.

\seref{s5} contains our main result, the Borcherds identity for $\ph$-twisted modules. In particular, as a consequence, we derive a commutator formula for the logarithmic fields in a twisted module. We prove that the Borcherds identity can replace the locality and $n$-th product identity in the definition of a twisted module. 

In \seref{s6}, we describe all twisted modules of affine and Heisenberg vertex algebras in terms of modules over certain twisted versions of the corresponding Lie algebras. We also determine the action of the Virasoro algebra. For the Heisenberg vertex algebra, all twisted irreducible highest-weight modules are constructed explicitly. 

Throughout the paper, $\z,\z_1,\z_2,\dots$ will be commuting formal variables, and we will use the notation $\z_{ij}=\z_i-\z_j$ and $x^{(k)}=x^k/k!$. All vector spaces will be over $\CC$. We denote by $\ZZ_+$ the set of non-negative integers.

\section{Preliminaries on vertex algebras}\label{s2}

In this section, we briefly review the basic definitions and properties of vertex algebras and their modules.
For more details, we refer to \cite{FLM, K2, FB, LL, KRR}. 

\subsection{Quantum fields}\label{sqf}

A \emph{(quantum) field} on a vector space $V$ is a linear map from $V$ to the space of Laurent series 
$V(\!(\z)\!) = V[[\z]][\z^{-1}]$. The space of all fields
\begin{equation*}\label{vert1}
\QF(V)=\Hom_\CC(V, V(\!(\z)\!))
\end{equation*}
is closed under the derivative $\d_\z$.
The composition $a(\z)b(\z)$ of two fields is not well defined in general. Instead, one considers the composition $a(\z_1)b(\z_2)$, which is a map from $V$ to $V(\!(\z_1)\!)(\!(\z_2)\!)$. Note that $V(\!(\z_1)\!)(\!(\z_2)\!)$ and $V(\!(\z_2)\!)(\!(\z_1)\!)$ are two different subspaces of $V[[\z_1^{\pm1},\z_2^{\pm1}]]$ whose intersection is $V(\!(\z_1,\z_2)\!) = V[[\z_1,\z_2]][\z_1^{-1},\z_2^{-1}]$. 

A pair of fields $a,b$ is called \emph{local} \cite{Go,DL,Li1} if
\begin{equation}\label{vert2}
\z_{12}^N \, a(\z_1) b(\z_2) = \z_{12}^N \, b(\z_2) a(\z_1) \,, \qquad \z_{12}=\z_1-\z_2 \,,
\end{equation}
for some integer $N\ge0$. When applied to any $v\in V$, both sides of this equation become elements of $V(\!(\z_1,\z_2)\!)$.
For $n\in\ZZ$, the \emph{$n$-th product} $a_{(n)}b$ of two local fields $a,b$ is defined by (cf.\ \cite{BN,BK2}):
\begin{equation*}\label{vert3}
(a_{(n)}b)(\z)v = \d_{\z_1}^{(N-1-n)} \bigl(\z_{12}^N \, 
a(\z_1) b(\z_2)v \bigr)\big|_{ \z_1=\z_2=\z }
\end{equation*}
for $v\in V$, $n\le N-1$, and $a_{(n)}b=0$ for $n\ge N$.
It is easy to show that this definition is equivalent to the one due to H.\ Li \cite{Li1} (see \cite[Lecture 14]{KRR}).
Note that $a_{(n)}b$ is again a field and is independent of the choice of $N$ satisfying \eqref{vert2}.
Moreover, if $c$ is another field local with $a$ and $b$, then $c$ is local with $a_{(n)}b$
(Dong's Lemma \cite{Li1,K2}; see \leref{llogf4} below).

The constant field $I$ equal to the identity operator is local with any other field $a$, and satisfies
\begin{equation*}%\label{vert4}
a_{(n)}I = 0 \,, \quad a_{(-n-1)}I = \d_\z^{(n)} a \,, \qquad n\ge 0\,.
\end{equation*}
Let $\V\subset\QF(V)$ be a \emph{local collection} of fields, i.e., such that every pair $a,b\in\V$ is local. We will assume that $I\in\V$.
By Dong's Lemma, the smallest subspace $\bar\V\subset\QF(V)$ containing $\V$ and closed under all $n$-th products is again a local collection. Then $\bar\V$ is also closed under $\d_\z$.

\subsection{Vertex algebras}\label{svert}

A \emph{vertex algebra} is a vector space $V$ (space of states), with a distinguished vector $\vac\in V$ (vacuum vector)
and a linear map $Y\colon V\to\QF(V)$ (state-field correspondence), such that $Y(\vac)=I$ and $Y(V)$ is a 
local collection of fields. 
The fields $Y(a)$ $(a\in V)$ are usually written as
\begin{equation*}\label{vert5}
Y(a,\z) = \sum_{n\in\ZZ} a_{(n)} \, \z^{-n-1} \,, \qquad
a_{(n)} \in \End(V) \,,
\end{equation*}
and the coefficients $a_{(n)}$ are called the \emph{modes} of $a$.
%
%The \emph{formal residue} $\Res_\z$ is defined as the coefficient of $\z^{-1}$, so that 
%\begin{equation*}%\label{vert6}
%a_{(n)} = \Res_\z \z^n Y(a,\z) \,.
%\end{equation*}
%
This endows $V$ with products $a_{(n)}b \in V$ for all $a,b\in V$, $n\in\ZZ$, and
the map $Y$ is a homomorphism for all of them:
\begin{equation}\label{vert7}
Y(a_{(n)}b,z) = Y(a,\z)_{(n)} Y(b,z) \,.
\end{equation}

The \emph{translation operator} $T\in\End(V)$ is defined by $Ta=a_{(-2)}\vac$. Then
\begin{equation*}%\label{vert8}
[T,Y(a,\z)] = \d_z Y(a,z) \,, \qquad a\in V \,.
\end{equation*}
A field $a(\z)$ with this property is called \emph{translation covariant}.
By the Kac Existence Theorem \cite{K2,DK}, every local collection $\V\subset\QF(V)$ of translation covariant fields
generates a vertex algebra structure on $V$, provided that $V$ is linearly spanned by $\vac$ and all coefficients of
\begin{equation*}%\label{vert9}
\quad a_1(\z_1) \cdots a_r(\z_r) \vac \,, \qquad r\geq1 \,, \; a_i \in\V \,.
\end{equation*}
Note that $\bar\V$, as defined above, is also a vertex algebra and the map $Y\colon V\to\bar\V$ is an
isomorphism \cite{Li1,K2,DK}.

For future use, recall that a \emph{derivation} of $V$ is a linear operator $\D$ on $V$ such that
\begin{equation*}%\label{vert8d}
\D(a_{(n)}b) = (\D a)_{(n)}b + a_{(n)} (\D b) \,, \qquad a,b\in V \,, \;\; n\in\ZZ\,.
\end{equation*}
The space $\Der(V)$ of all derivations is a Lie algebra containing $T$.

\subsection{Borcherds identity}\label{sborid}

The main identity satisfied by the modes is the \emph{Borcherds identity}
(also called Jacobi identity \cite{FLM}):
\begin{equation}\label{vert10}
\begin{split}
\sum_{i=0}^\infty (-1)^i & \binom{n}{i} 
\Bigl( 
a_{(m+n-i)}(b_{(k+i)}c)
- (-1)^n \, b_{(k+n-i)}(a_{(m+i)}c)
\Bigr)
\\ 
&=
\sum_{j=0}^\infty \binom{m}{j} (a_{(n+j)}b)_{(m+k-j)}c \,,
\end{split}
\end{equation}
where $a,b,c \in V$. Observe that the above sums are finite, because
$a_{(j)}b = 0$ for sufficiently large $j$.
In particular, setting $n=0$ in the Borcherds identity, we obtain the \emph{commutator formula}
\begin{equation}\label{vert11}
[a_{(m)}, b_{(k)}] = \sum_{j=0}^\infty \binom{m}{j} (a_{(j)}b)_{(m+k-j)}c \,.
\end{equation}
Equivalently,
\begin{equation}\label{vert12}
[Y(a,\z_1), Y(b,\z_2)] = \sum_{j=0}^\infty Y(a_{(j)}b,\z_2) \, \d_{\z_2}^{(j)} \de(\z_1,\z_2) \,,
\end{equation}
where
\begin{equation*}%\label{vert13}
\de(\z_1,\z_2) = \sum_{m\in\ZZ} \z_1^{-m-1} \z_2^m
\end{equation*}
is the formal \emph{delta function}. It is often convenient to use the formal expansions
\begin{equation}\label{vert14}
\begin{split}
\io_{\z_1,\z_2} \z_{12}^n &= \sum_{i=0}^\infty \binom{n}{i} (-1)^i \z_1^{n-i} \z_2^i \,, \\
%\in \CC (\!(\z_1)\!) (\!(\z_2)\!) \,, \\
\io_{\z_2,\z_1} \z_{12}^n &= \sum_{i=0}^\infty \binom{n}{i} (-1)^{n+i} \z_1^i \z_2^{n-i} \,.
%\in \CC (\!(\z_2)\!) (\!(\z_1)\!) \,.
\end{split}
\end{equation}
Then
\begin{equation}\label{vert15}
\d_{\z_2}^{(j)} \de(\z_1,\z_2) = (\io_{\z_1,\z_2} - \io_{\z_2,\z_1}) \z_{12}^{-j-1} \,, \qquad j\ge0 \,.
\end{equation}
The delta function has the property
\begin{equation}\label{vert16}
\Res_{\z_1} a(\z_1) \, \d_{\z_2}^{(j)} \de(\z_1,\z_2) = \d_{\z_2}^{(j)} a(\z_2)
\end{equation}
for any field $a(\z)$, where as usual $\Res_\z$ denotes the coefficient of $\z^{-1}$.

\subsection{Twisted modules}\label{stwmod}

A \emph{representation} (or \emph{module}) of $V$ is a vector space $W$ endowed with a
linear map $Y\colon V\to\QF(W)$ such that $Y(\vac)=I$ and the Borcherds identity
\eqref{vert10} holds for $a,b\in V$, $c\in W$ (see \cite{FB, LL,KRR}).
Equivalently, due to \cite{Li1}, one can replace \eqref{vert10} by the condition that $Y(V)\subset\QF(W)$ is a local collection
of fields satisfying the $n$-th product identity \eqref{vert7}. The commutator formulas \eqref{vert11}, \eqref{vert12}
hold for modules as well.

Recall that an \emph{automorphism} of a vertex algebra $V$ is an invertible linear operator $\si$ on $V$ such that
\begin{equation*}%\label{vert17}
\si( a_{(n)} b ) = (\si a)_{(n)} (\si b) \,, \qquad a,b\in V \,, \;\; n\in\ZZ \,.
\end{equation*}
The group of all automorphisms of $V$ is denoted $\Aut(V)$.
If $\si\in\Aut(V)$ has a finite order $r$, then $\si$ is semisimple with eigenvalues $r$-th roots of $1$.
In the definition of a \emph{$\si$-twisted representation} $W$ of $V$, the image of the
above map $Y$ is allowed to have non-integral (rational) powers of $\z$ (see \cite{FFR,D,KRR}).
More precisely,
\begin{equation*}%\label{vert18}
Y(a,\z) = \sum_{n\in p+\ZZ} a_{(n)} \, \z^{-n-1} \,, \qquad
\text{if} \quad \si a = e^{-2\pi\ii p} a \,, \;\; p\in\frac1r\ZZ \,,
\end{equation*}
where $a_{(n)} \in \End(W)$.
Equivalently, the monodromy around $\z=0$ is given by the action of $\si$:
\begin{equation*}%\label{vert19}
Y(\si a,\z) = Y(a, e^{2\pi\ii}\z) \,, \qquad a\in V \,.
\end{equation*}
The Borcherds identity \eqref{vert10} satisfied by the modes remains the
same in the twisted case, provided that 
\begin{equation*}%\label{vert20}
\si a = e^{-2\pi\ii m} a \,, \quad  \si b = e^{-2\pi\ii k} b \,, \qquad m,k\in\QQ \,, \;\; n\in\ZZ\,.
\end{equation*}
As a consequence, we also have the commutator formula \eqref{vert11}. However, \eqref{vert12}
needs to be modified for twisted modules (see, e.g., \cite{BK1} and \eqref{twlog16} below).
It was proved in \cite{BM} that in the definition of a twisted module the Borcherds identity
can be replaced by the locality of all $Y(a,z)$ and the $n$-th product identity \eqref{vert7}.
Note that in the twisted case our definition of $n$-th product differs from H.\ Li's one from \cite{Li1,Li2}.

\section{Logarithmic quantum fields}\label{s3}

In this section, we introduce the notions of a logarithmic field, locality, and $n$-th products of logarithmic fields. We express the $n$-th product in terms of the normally ordered product and propagator. We prove that any local collection of logarithmic fields generates a vertex algebra.

\subsection{Logarithmic fields and locality}\label{slogf}

As before,  $\z,\z_1,\z_2,\dots$ will be formal variables, and let $\ze,\ze_1,\ze_2,\dots$ be another set of
formal variables corresponding to them, which will be thought of as $\ze=\log\z$ and $\ze_i=\log\z_i$. More precisely,
instead of $\d_\z$ and $\d_\ze$, we will work with the derivations
\begin{equation*}%\label{logf1}
D_\z = \d_\z+\z^{-1} \d_\ze \,, \qquad D_\ze =\z\d_\z+\d_\ze  \,,
\end{equation*}
and similarly for $D_{\z_i}$, $D_{\ze_i}$.

Fix a vector space $W$ over $\CC$. For $\al\in\CC/\ZZ$, 
we denote by $W[\ze][[\z]] \z^{-\al}$ the space of all formal series of the form (cf.\ \cite{BK2}):
\begin{equation*}%\label{logf2}
\sum_{i=0}^\infty w_{i}(\ze) \z^{i-m} \,, \qquad w_{i}(\ze)\in W[\ze] \,, \;\; m\in\al \,.
\end{equation*}
For example, $W[\ze][[\z]] \z^\ZZ = W[\ze](\!(\z)\!)$ is the space of Laurent series in $\z$ with coefficients in $W[\ze]$.
Observe that $W[\ze][[\z]] \z^{-\al}$ is a module over the ring $\CC(\!(\z)\!)$, and is closed under the derivations 
$D_\z$ and $D_\ze$.

\begin{definition}\label{dlogf1}
With the above notation, let
\begin{equation*}%\label{logf4}
\LF_\al(W) = \Hom_\CC(W,W[\ze][[\z]] \z^{-\al}) \,, \qquad \al\in\CC/\ZZ \,,
\end{equation*}
and
\begin{equation*}%\label{logf5}
\LF(W) = \bigoplus_{\al\in\CC/\ZZ} \LF_\al(W) \,.
\end{equation*}
The elements of $\LF(W)$ are called \emph{logarithmic (quantum) fields} on $W$, and are denoted as
$a(\ze,z)$ or $a(\z)$ for short.
\end{definition}

By definition, every logarithmic field $a(\z)$ is a finite sum of elements from the spaces $\LF_\al(W)$.
The composition of two logarithmic fields $a\in\LF_\al(W)$ and $b\in\LF_\be(W)$ is the linear map
\begin{equation*}%\label{logf6}
a(\z_1)b(\z_2) \colon W\to \bigl( W[\ze_1][[\z_1]] \z_1^{-\al} \bigr) [\ze_2][[\z_2]] \z_2^{-\be} \,.
\end{equation*}

\begin{definition}\label{dlogf2}
A pair of logarithmic fields $a,b$ is called \emph{local} if
\begin{equation}\label{logf7}
\z_{12}^N \, a(\z_1) b(\z_2) = \z_{12}^N \, b(\z_2) a(\z_1) \,, \qquad \z_{12}=\z_1-\z_2 \,,
\end{equation}
for some integer $N\ge0$.
\end{definition}

For every $v\in W$, the powers of $\z_2$ in $a(z_1)b(z_2)v$ belong to the union of finitely many sets of the form $\ga+\ZZ_+$ ($\ga\in\CC$). If $a(\z_1)$ and $b(\z_2)$ are local, then
\begin{equation*}%\label{logf8}
\z_{12}^N \, a(\z_1) b(\z_2)v = \z_{12}^N \, b(\z_2) a(\z_1)v
\end{equation*}
satisfies this property both for the powers of $\z_1$ and $\z_2$. In fact, when $a\in\LF_\al(W)$ and $b\in\LF_\be(W)$ are local, both sides of this equation belong to the space
\begin{equation*}%\label{logf2ab}
W[\ze_1,\ze_2][[\z_1,\z_2]] \z_1^{-\al} \z_2^{-\be} \,.
\end{equation*}

\subsection{$n$-th products}\label{slocnpr}

Now we define an operation on local logarithmic fields, which provides an algebraic formulation of the
operator product expansion (cf.\ \cite{BN,BK2,BM}).

\begin{definition}\label{dlogf3}
For $n\in\ZZ$, the \emph{$n$-th product} $a_{(n)}b$ of two local logarithmic fields $a,b$ is defined by:
\begin{equation}\label{logf9}
(a_{(n)}b)(\ze,\z)v = D_{\z_1}^{(N-1-n)} \bigl(\z_{12}^N \, 
a(\ze_1,\z_1) b(\ze_2,\z_2)v \bigr)\Big|_{ \substack{\z_1=\z_2=\z \\ \ze_1=\ze_2=\ze} }
\end{equation}
for $v\in W$ and $n\le N-1$. For $n\ge N$, let $a_{(n)}b=0$. As before, we will supress the dependence on $\ze$ and understand that setting $\z_1=z$ automatically sets $\ze_1=\ze$.
\end{definition}

Note that $a_{(n)}b$ is again a logarithmic field, and it does not depend on the choice of $N$ satisfying \eqref{logf7}.
Moreover, $a_{(n)}b \in \LF_{\al+\be}(W)$
if $a\in\LF_\al(W)$ and $b\in\LF_\be(W)$. 
Using the Leibniz rule, one can derive from \eqref{logf9} the following properties:
\begin{align*}
(D_\z a)_{(n)}b &= -n a_{(n-1)}b \,, \\
D_\z(a_{(n)}b) &= (D_\z a)_{(n)}b + a_{(n)} (D_\z b) \,, \\
\d_\ze(a_{(n)}b) &= (\d_\ze a)_{(n)}b + a_{(n)} (\d_\ze b) \,.
\end{align*}
We also have an analog of Dong's Lemma (cf.\ \cite{Li1,Li2,K2}).

\begin{lemma}\label{llogf4}
Let\/ $a,b,c$ be logarithmic fields such that the pairs\/ $(a,b)$, $(a,c)$, $(b,c)$ are local. Then\/ $a_{(n)}b$ and\/ $c$ are local for all\/ $n\in\ZZ$.
\end{lemma}
\begin{proof}
For some sufficiently large $N$, we have \eqref{logf7} and
\begin{equation*}
\z_{13}^N \, a(\z_1) c(\z_3) = \z_{13}^N \, c(\z_3) a(\z_1) \,, \qquad
\z_{23}^N \, b(\z_2) c(\z_3) = \z_{23}^N \, c(\z_3) b(\z_2) \,.
\end{equation*}
Using the Leibniz rule, we find for $n'=N-1-n\ge0$ and $v\in W$,
\begin{align*}
\z_{23}^{2N+n'} (a_{(n)}b)(\z_2) c(\z_3)v 
= \Bigl(\z_{13}^{N+n'} D_{\z_1}^{(n')} \bigl(\z_{12}^N \, \z_{23}^N \, a(\z_1) b(\z_2) c(\z_3)v \bigr)\Bigr)&\Big|_{\z_1=\z_2} \\
= \sum_{i=0}^{n'} (-1)^{n'-i} \binom{N+n'}{n'-i} D_{\z_1}^{(i)} 
\Bigl(\z_{13}^{N+i} \, \z_{12}^N \, \z_{23}^N \, a(\z_1) b(\z_2) c(\z_3)v \Bigr)&\Big|_{\z_1=\z_2} \,.
\end{align*}
We can move $c(\z_3)$ to the left of $a(\z_1) b(\z_2)$ inside the parentheses, and then rewrite the whole expression back as
$\z_{23}^{2N+n'} c(\z_3) (a_{(n)}b)(\z_2) v$.
\end{proof}

\subsection{Normally ordered products and propagators}\label{sprop}

For $\al\in\CC/\ZZ$, pick the unique representative $\al_0\in\al$ with $-1<\Re\al_0\le0$.
Every logarithmic field $a\in\LF_\al(W)$ can be expanded as
\begin{equation*}%\label{logf10}
a(\ze,\z) = \sum_{i\in\ZZ} a_i(\ze) \z^{-i-\al_0} \,, \qquad a_i(\ze) \in\Hom_\CC(W,W[\ze]) \,,
\end{equation*}
where for each $v\in W$ we have $a_i(\ze)v=0$ for sufficiently large $i$.
The \emph{annihilation} and \emph{creation parts} of $a(\z)$ are defined respectively as
\begin{align*}%\label{logf11}
a(\z)_- &= a(\ze,\z)_- = \sum_{i=1}^\infty a_i(\ze) \z^{-i-\al_0} \,, \\
%\label{logf12}
a(\z)_+ &= a(\ze,\z)_+ = \sum_{i=-\infty}^{0} a_i(\ze) \z^{-i-\al_0} \,.
\end{align*}
These are extended by linearity to all $a\in\LF(W)$.
In other words, $a(\z)_-$ is the part of $a(\z)$ containing only $\z^\ga$ with $\Re\ga<0$, while
in $a(\z)_+$ we have only $\z^\ga$ with $\Re\ga\ge0$.

\begin{definition}\label{dlogf4}
The \emph{normally ordered product} of two logarithmic fields $a(\z_1)$, $b(\z_2)$ is defined by:
\begin{equation*}%\label{logf12}
{:} a(\z_1) b(\z_2) {:} = a(\z_1)_+ b(\z_2) + b(\z_2) a(\z_1)_- \,.
\end{equation*}
Their \emph{propagator} is:
\begin{equation*}%\label{logf13}
P( a, b; \z_1,\z_2) = [a(\z_1)_- , b(\z_2)] = a(\z_1) b(\z_2) - {:} a(\z_1) b(\z_2) {:} \,.
\end{equation*}
\end{definition}

Just like for usual quantum fields (see, e.g., \cite{K2}),
it is easy to check that ${:} a(\z_1) b(\z_2) {:}$ is well defined for $\z_1=\z_2$ and 
${:} a(\z) b(\z) {:}$ is again a logarithmic field.

\begin{proposition}\label{ptlogf5}
Let\/ $a$ and\/ $b$ be two local logarithmic fields, and\/ $N$ be from\/ \eqref{logf7}. Then
for\/ $0\le n\le N-1$ and\/ $k\ge0$, we have$:$
\begin{align*}%\label{logf14}
(a_{(n)}b)(\z) &= D_{\z_1}^{(N-1-n)} \bigl(\z_{12}^N \, P( a, b; \z_1,\z_2) \bigr)\big|_{ \z_1=\z_2=\z} \,, \\
%\label{logf15}
(a_{(-k-1)}b)(\z) &= {:} \bigl( D_\z^{(k)} a(\z) \bigr) b(\z) {:} 
+ D_{\z_1}^{(N+k)} \bigl(\z_{12}^N \, P( a, b; \z_1,\z_2) \bigr)\big|_{ \z_1=\z_2=\z} \,.
\end{align*}
\end{proposition}
\begin{proof}
The proof is a straightforward calculation using
\begin{equation*}
\z_{12}^N \, a(\z_1) b(\z_2) = \z_{12}^N \, {:} a(\z_1) b(\z_2) {:} + \z_{12}^N \, P( a, b; \z_1,\z_2)
\end{equation*}
and the fact that ${:} a(\z_1) b(\z_2) {:}$ is well defined for $\z_1=\z_2$.
\end{proof}

\subsection{Local collections of logarithmic fields}\label{sloccol}

The identity operator $I$ is local with any other logarithmic field $a$, and satisfies
\begin{equation*}%\label{logf16}
a_{(n)}I = 0 \,, \quad a_{(-n-1)}I = D_\z^{(n)} a \,, \qquad n\ge 0\,.
\end{equation*}
Let $\W\subset\LF(W)$ be a \emph{local collection}, i.e., such that every pair $a,b\in\W$ is local. We can add $I$ to $\W$ and still have a local collection. If a pair $(a,b)$ is local, then $(D_\ze a,b)$ is also local; thus the $\CC[D_\ze]$-module generated by $\W$ is again  local. Due to \leref{llogf4}, the smallest subspace $\bar\W\subset\LF(W)$ containing $\W\cup\{I\}$ and closed under $D_\ze$ and all $n$-th products is a local collection. Similarly to  \cite{Li1,Li2}, we have the following result.

\begin{theorem}\label{tlogf5}
The $n$-th products endow the space\/ $\bar\W$ with the structure of a vertex algebra with a vacuum vector $I$ and translation operator\/ $D_\z$.
\end{theorem}
\begin{proof}
The state-field correspondence $Y\colon\bar\W\to\QF(\bar\W)$ is given by 
\begin{equation*}
Y(a,x)b = \sum_{n\in\ZZ} x^{-n-1} (a_{(n)}b)  \,, \qquad a,b\in\bar\W \,,
\end{equation*}
where the formal variable is now denoted by $x$, and $a_{(n)}b$ is again defined by \eqref{logf9}. 
Due to the already established properties of the $n$-th products, it only remains to prove that $Y(a,x)$ and $Y(b,x)$ are local for $a,b\in\bar\W$.

When we need to specify the formal variable in the fields $a$ and $b$, we will write $Y$ as
\begin{equation*}
Y\bigl(a(\z),x\bigr)b(\z) = \sum_{n\in\ZZ} x^{-n-1} (a_{(n)}b)(\z) \,.
\end{equation*}
It follows immediately from \eqref{logf9} that for all $v\in W$,
\begin{equation*}
\bigl( Y\bigl(a(\z),x\bigr)b(\z) \bigr) v = x^{-N} e^{x D_{\z_1}} \bigl( \z_{12}^N \, a(\z_1) b(\z_2)v \bigr)\big|_{\z_1=\z_2=\z} \,,
\end{equation*}
where, as before, $N$ is such that the locality  \eqref{logf7} holds. For brevity, through the rest of the proof we will omit the vector $v$.

Consider $a,b,c\in\bar\W$, and take $N$ to be an even number such that \eqref{logf7} holds for the pairs $(a,b)$, $(a,c)$ and $(b,c)$. By the proof of \leref{llogf4}, for any $k\le N-1$, the pair
$(a,b_{(k)}c)$ satisfies \eqref{logf7} with $N$ replaced by $2N+k'$ where $k'=N-1-k\ge0$. Therefore,
\begin{align*}
Y\bigl(& a(\z), x_1\bigr) (b_{(k)}c)(\z) \\
&= x_1^{-2N-k'} e^{x_1 D_{\z_1}} \Bigl( \z_{13}^{2N+k'} D_{\z_2}^{(k')} \bigl( \z_{23}^N \, a(\z_1) b(\z_2) c(\z_3) \bigr) \Bigr) \Big|_{\z_1=\z_2=\z_3=\z} \,.
\end{align*}
Summing over $k$, we obtain:
\begin{align*}
& x_1^{2N} x_2^{N} \, Y\bigl(a(\z), x_1\bigr) Y\bigl(b(\z), x_2\bigr) c(\z)  \\
&= \sum_{k'=0}^\infty \Bigl(\frac{x_2}{x_1}\Bigr)^{k'} e^{x_1 D_{\z_1}} \Bigl( \z_{13}^{2N+k'} D_{\z_2}^{(k')} \bigl( \z_{23}^N \, a(\z_1) b(\z_2) c(\z_3) \bigr) \Bigr) \Big|_{\z_1=\z_2=\z_3=\z} \,.
\end{align*}
Notice that here $\z_{13}^{2N+k'}$ can be replaced by $\z_{13}^{N} \z_{12}^{N+k'}$. 
Consider the linear operator 
\begin{equation*}
A_2 = \sum_{k'=0}^\infty \Bigl(\frac{x_2}{x_1}\Bigr)^{k'} \z_{12}^{k'} \, D_{\z_2}^{(k')} = \Bigl(1-\frac{x_2}{x_1}\Bigr)^{\z_{21} D_{\z_2}} \,,
\end{equation*}
where we applied the well-known identity $\z^k\d_{\z}^{(k)} = \binom{\z\d_\z}{k}$.

Then
\begin{equation*}
\z_{12}^{N} \, A_2 = \Bigl(1-\frac{x_2}{x_1}\Bigr)^{-N} A_2 \circ \z_{12}^{N} \,,
\end{equation*}
and we have
\begin{align*}
x_1^{2N} & x_2^{N} \, Y\bigl(a(\z), x_1\bigr) Y\bigl(b(\z), x_2\bigr) c(\z)  \\
&= e^{x_1 D_{\z_1}} \bigl( \z_{13}^{N} \, \z_{12}^{N} \, A_2 \bigl( \z_{23}^N \, a(\z_1) b(\z_2) c(\z_3) \bigr) \bigr) \big|_{\z_1=\z_2=\z_3=\z} \\
&=  \Bigl(1-\frac{x_2}{x_1}\Bigr)^{-N} e^{x_1 D_{\z_1}} A_2 \bigl( \z_{13}^{N} \, \z_{12}^{N} \, \z_{23}^N \, a(\z_1) b(\z_2) c(\z_3) \bigr) \big|_{\z_1=\z_2=\z_3=\z}\,.
\end{align*}
Now observe that
\begin{equation*}
e^{x_1 D_{\z_1}}  A_2 = \sum_{k=0}^\infty \Bigl(\frac{x_2}{x_1}\Bigr)^{k} (x_1+\z_{12})^{k} \, D_{\z_2}^{(k)} e^{x_1 D_{\z_1}}
\end{equation*}
becomes $e^{x_2 D_{\z_2}} e^{x_1 D_{\z_1}}$ after setting $\z_1=\z_2$. Therefore,
\begin{align*}
x_1^{N} & x_2^{N} (x_1-x_2)^N \, Y\bigl(a(\z), x_1\bigr) Y\bigl(b(\z), x_2\bigr) c(\z)  \\
&= e^{x_1 D_{\z_1} + x_2 D_{\z_2}} \bigl( \z_{13}^{N} \, \z_{12}^{N} \, \z_{23}^N \, a(\z_1) b(\z_2) c(\z_3) \bigr) \big|_{\z_1=\z_2=\z_3=\z}\,.
\end{align*}
This implies the locality of $Y(a,x)$ and $Y(b,x)$, thus completing the proof of the theorem.
\end{proof}

It follows from \eqref{logf9} and $[D_\ze,D_\z]=-D_z$ that
\begin{equation*}%\label{logf17}
D_\ze(a_{(n)}b) = (D_\ze a)_{(n)}b + a_{(n)} (D_\ze b) + (n+1)(a_{(n)}b) \,.
\end{equation*}
Hence, $e^{2\pi\ii D_\ze}$ is an automorphism of the vertex algebra $\bar\W$. It acts exactly as the monodromy operator around $0$, sending $\ze$ to $\ze+2\pi\ii$ and $\z^\ga$ to $e^{2\pi\ii\ga} \z^\ga$.
Note that $e^{2\pi\ii D_\ze} = e^{2\pi\ii \z\d_\z} e^{2\pi\ii \d_\ze}$ and $e^{2\pi\ii \z\d_\z} \in\Aut(\bar\W)$, $\d_\ze\in\Der(\bar\W)$.

\section{Definition of twisted modules}\label{s4}

From now on, $V$ will be a vertex algebra and $\ph$ an automorphism of $V$, which is not necessarily of finite order.
In this section, we introduce the notion of a $\ph$-twisted $V$-module and establish some of its basic properties.
We continue to use the notation from \seref{s3}.

\subsection{$\ph$-twisted modules}\label{sphtwm}

The following is the main object of the paper.

\begin{definition}\label{dltwlog2}
A \emph{$\ph$-twisted $V$-module} is a vector space $W$, equipped with a linear map
$Y\colon V\to\PLF(W)$ such that $Y(\vac)=I$ is the identity operator, $Y(V)$ is a local collection,
\begin{equation}\label{twlog2}
Y(\ph a,\z) = e^{2\pi\ii D_\ze} Y(a, \z) \,,
\end{equation}
and 
\begin{equation}\label{twlog3}
Y(a_{(n)}b,\z)=Y(a,\z)_{(n)} Y(b,\z)
\end{equation}
for all $a,b\in V$, $n\in\ZZ$. %Sometimes, we will denote such modules more explicitly as $(W,Y)$.
We will call \eqref{twlog2} the \emph{$\ph$-equivariance}, and \eqref{twlog3} the \emph{$n$-th product identity}.
\end{definition}

\begin{remark}\label{rtwlog1}
Y.-Z.~Huang has introduced in \cite{H} a notion of a $\ph$-twisted $V$-module $W$, which is more restrictive than ours (in particular, it assumes certain gradings of $V$ and $W$). One can show that every $\ph$-twisted module in the sense of \cite{H} satisfies our definition.
Conversely, as will be indicated below, some assumptions of \cite[Definition 3.1]{H} also hold in our case.
\end{remark}

\begin{remark}\label{rtwlog2}
S.-Q.~Liu, D.~Yang, and Y.~Zhang have introduced in \cite{LYZ} a notion of a $\ph$-twisted $V$-module, which has some similarities to ours but also important differences. In particular, it involves vectors in $V\otimes\CC^d$ and a certain
$d\times d$ matrix associated to a Frobenius manifold of dimension~$d$.
%They only formulate it in the case when $V$ is the Heisenberg vertex algebra (see \seref{s6} below).
\end{remark}

As a consequence of \eqref{twlog3} and $Ta=a_{(-2)}\vac$, we have (cf.\ \cite{H}):
\begin{equation}\label{twlog3t}
Y(Ta,\z)=D_\z Y(a,\z) \,, \qquad a\in V\,.
\end{equation}
Eqs.\ \eqref{twlog2}, \eqref{twlog3} can be stated equivalently that $Y\colon V\to\bar\W$ is a vertex algebra homomorphism compatible with
the automorphisms $\ph$ and $e^{2\pi\ii D_\ze}$, where $\bar\W=Y(V)$ (see \thref{tlogf5}). 

\begin{example}\label{etwlog}
Let $W$ be a vector space, $\W\subset\LF(W)$ a local collection, $\bar\W$ be the vertex algebra generated by $\W$, and 
$\ph=e^{2\pi\ii D_\ze} \in\Aut(\bar\W)$ (see \thref{tlogf5}). 
Then the identity map $Y\colon\bar\W\to\LF(W)$, $Y(a,\z)=a(\z)$, provides $W$ with the structure of a $\ph$-twisted $\bar\W$-module
(cf.\ \cite{Li1,Li2}).
\end{example}

\begin{remark}\label{rtwlog3}
The space $V^\ph$ of $\ph$-invariants (i.e., $a\in V$ such that $\ph a=a$) is a subalgebra of $V$.
The restriction of any $\ph$-twisted $V$-module to $V^\ph$ is a (untwisted) $V^\ph$-module.
\end{remark}

\subsection{Locally finite automorphisms}\label{slocfin}

A linear operator $\ph$ on $V$ is called \emph{locally finite} if every $a\in V$ is contained in some
finite-dimensional $\ph$-invariant subspace of $V$ (see \cite[Chapter 3]{K1}). 
In particular, this holds when $V$ is a direct sum of finite-dimensional $\ph$-invariant subspaces, as is assumed in \cite{H}.
A linear operator $\N$ on $V$ is called \emph{locally nilpotent} if for every $a\in V$ we have $\N^l a=0$ for some $l\ge1$.
The next lemma is standard.

\begin{lemma}\label{ltwlog1}
Every invertible locally finite linear operator\/ $\ph$ can be written uniquely in the form\/ $\ph=\si e^{-2\pi\ii\N}$, where\/ $\si$
is semisimple, $\N$ is locally nilpotent and\/ $\si\N=\N\si$. Furthermore, if\/ 
$\ph\in\Aut(V)$, then\/ $\si\in\Aut(V)$ and\/ $\N\in\Der(V)$.
\end{lemma}
\begin{proof}
Fix $a\in V$ and a finite-dimensional subspace $U\subset V$ such that $a\in U$ and $\ph(U)\subset U$. Then the restriction $\ph|_U$
has the desired decomposition (Jordan--Chevalley decomposition). If we have another such subspace $U'\supset U$, it will give rise to the same $\si$ and $\N$ when restricted to $U$. Therefore, $\si$ and $\N$ are uniquely defined and are independent of the choice of $U$.

Let $\ph\in\Aut(V)$, and $a,b\in V$ be such that $(\ph-\la)^l a = (\ph-\mu)^l b = 0$ for some $\la,\mu\in\CC$ and $l\ge1$. Then
$\si a = \la a$ and $\si b = \mu b$. The identity
\begin{equation*}
\ph\otimes\ph-\la\otimes\mu = (\ph-\la)\otimes\ph +\la\otimes(\ph-\mu)
\end{equation*}
then implies that
\begin{equation*}
(\ph-\la\mu)^{m} (a_{(n)}b) = \sum_{k=0}^{m} \binom{m}{k} \bigl( (\ph-\la)^k \la^{m-k} a \bigr)_{(n)} \bigl( \ph^k (\ph-\mu)^{m-k} b \bigr)
=0
\end{equation*}
for $m\ge 2l-1$. Therefore, 
$\si(a_{(n)}b) = \la\mu(a_{(n)}b)$ and $\si\in\Aut(V)$.

To prove that $\N\in\Der(V)$, consider the expression
\begin{equation*}
e^{2\pi\ii x\N} (a_{(n)}b) - (e^{2\pi\ii x\N} a)_{(n)} (e^{2\pi\ii x\N} b) \,,
\end{equation*}
which is a polynomial in $x$. This polynomial vanishes at all $k\in\ZZ$, since $(\ph\si^{-1})^k \in\Aut(V)$.
Taking $\d_x$ at $x=0$, we obtain that $\N\in\Der(V)$.
\end{proof}

We will say that $\ph$ is \emph{locally finite on} $a\in V$ if there is a finite-dimensional subspace $U\subset V$ such that $a\in U$ and $\ph(U)\subset U$.

\begin{lemma}\label{lltwlog2}
The set\/ $\bar V$ of all\/ $a\in V$, on which\/ $\ph$ is locally finite, is the
maximal\/ $\ph$-invariant subspace\/ $\bar V\subset V$ such that the restriction\/ 
$\ph|_{\bar V}$ is locally finite. Moreover, $\bar V$ is a subalgebra of\/ $V$.
\end{lemma}
\begin{proof}
This follows easily from the definitions. Indeed, let $U$ and $U'$ be finite-dimensional $\ph$-invariant subspaces
such that $a\in U$, $b\in U'$. Then $a+\la b\in U+U'$ and $a_{(n)} b \in U_{(n)} U'$ for all $\la\in\CC$, $n\in\ZZ$.
\end{proof}

\subsection{Consequences of local finiteness}\label{sconlf}

Consider again an automorphism $\ph$ of $V$ and a $\ph$-twisted $V$-module $W$.
Let $\bar V\subset V$ be the maximal subalgebra on which $\ph$ is locally finite (see \leref{lltwlog2}).
Write $\ph|_{\bar V}=\si e^{-2\pi\ii\N}$, as in \leref{ltwlog1}, where
$\si\in\Aut(\bar V)$ is semisimple and $\N\in\Der(\bar V)$ is locally nilpotent.

\begin{lemma}\label{lltwlog2b}
For all\/ $a\in\bar V$, we have
\begin{equation}\label{twlog2a}
Y(\si a,\z) = e^{2\pi\ii\z\d_\z} Y(a, \z) \,, \qquad 
Y(\N a,\z) = -\d_\ze Y(a, \z) \,.
\end{equation}
\end{lemma}
\begin{proof}
This follows from the uniqueness of $\si$ and $\N$ from \leref{ltwlog1}, since
the semisimple part of $e^{2\pi\ii D_\ze}$
is $e^{2\pi\ii \z\d_\z}$ and the corresponding locally nilpotent operator is $-\d_\ze$.
\end{proof}

In particular, since $\N$ is locally nilpotent,
we see from \eqref{twlog2a} that each logarithmic field $Y(a,\z)$ is a polynomial in $\ze$ 
for $a\in\bar V$ (cf.\ \cite{H}).
Introduce the linear map $\z^\N$ from $\bar V$ to $\bar V[\ze]$, given by
\begin{equation*}%\label{twlog5}
\z^\N a = e^{\ze\N} a \in \bar V[\ze] \,, \qquad a\in \bar V \,,
\end{equation*}
and let
\begin{equation*}%\label{twlog6}
X(a,\z) = Y(\z^\N a,\z) \,, \quad Y(a,\z) = X(\z^{-\N} a,\z) \,, \qquad a\in \bar V \,.
\end{equation*}
Note that, since $\N\in\Der(\bar V)$, we have
\begin{equation}\label{twlog7}
\z^\N(a_{(n)}b) = (\z^\N a)_{(n)} (\z^\N b) \,, \qquad a,b\in \bar V \,, \;\; n\in\ZZ\,.
\end{equation}

\begin{lemma}\label{lltwlog3}
With the above notation, we have\/ $X(a,\z) = Y(a,\z)|_{\ze=0}$ for all\/ $a\in \bar V$. Furthermore, $X(\bar V)$ is a
local collection of fields.
\end{lemma}
\begin{proof}
Using \eqref{twlog2a}, we find
\begin{align*}%\label{twlog8}
\d_\ze X(a,\z) &= Y(\N\z^\N a,\z) +\d_\ze Y(a',\z)\big|_{a'=\z^\N a}  \\
&= Y(\N\z^\N a,\z) - Y(\N a',\z)\big|_{a'=\z^\N a} = 0 \,.
\end{align*}
Then $X(a,\z) = X(a,\z)|_{\ze=0} = Y(a,\z)|_{\ze=0}$.
Setting $\ze_1=\ze_2=0$ in \eqref{logf7}, we see that all the fields $X(a,\z)$ are local.
\end{proof}

%However, we will see in the next section that the fields $X(a,\z)$ do not satisfy the $n$-th product identity %\eqref{twlog3} in general. 

\begin{remark}\label{rtwlog4}
The kernel $V^\N \subset\bar V$ of $\N$ is a subalgebra of $V$. The restriction of any $\ph$-twisted $V$-module to $V^\N$ is a $\si$-twisted $V^\N$-module.
\end{remark}

\subsection{$\D$-twisted modules}\label{ssintw}

When $\N$ is not locally nilpotent, we might not be able to exponentiate it. However, \eqref{twlog2a} still makes sense, suggesting the following more general notion, which we plan to investigate in the future.

\begin{definition}\label{dltwlog3}
Let $V$ be a vertex algebra and $\D\in\Der(V)$.
A \emph{$\D$-twisted $V$-module} is a vector space $W$, equipped with a linear map
\begin{equation*}%\label{dtwmod1}
Y\colon V\to \Hom_\CC \bigl( W, W[[\ze]](\!(\z)\!) \bigr)
\end{equation*}
such that 
\begin{equation*}%\label{dtwmod2}
Y(\D a,\z) = -\d_\ze Y(a,\z) \,, \qquad a\in V \,,
\end{equation*}
$Y(\vac)=I$ is the identity operator, $Y(V)$ is a local collection, and the $n$-th product identity \eqref{twlog3} holds.
\end{definition}

Observe that, in comparison to \deref{dltwlog2},  here we allow the logarithmic fields $Y(a,\z)$ to have coefficients formal power series in $\ze$. However, all results of \seref{s3} still hold in this case.
When $\D$ is locally finite, then $\ph=e^{-2\pi\ii\D} \in\Aut(V)$ is locally finite and the notion of a $\D$-twisted module is equivalent to that of a $\ph$-twisted module.

\begin{remark}\label{rtwlog5}
For every $\D\in\Der(V)$, the kernel $V^\D$ of $\D$ is a subalgebra of $V$. The restriction of any $\D$-twisted $V$-module to $V^\D$ is a (untwisted) $V^\D$-module.
\end{remark}

%With obvious modifications, one can also define the notion of a \emph{$(\si,\D)$-twisted $V$-module}, where
%$\si\in\Aut(V)$ is semisimple and $\si\D=\D\si$. When $\N$ is locally nilpotent, then $\ph=\si e^{-2\pi\ii\N} %\in\Aut(V)$ is locally finite. In that case the notion of a $(\si,\N)$-twisted module is equivalent to that of a 
%$\ph$-twisted module.

\section{Borcherds identity for twisted modules}\label{s5}

In this section, we derive a Borcherds identity for twisted modules and, as a consequence, a commutator formula. We prove that the Borcherds identity can replace the locality and $n$-th product identity in the definition of a twisted module.

\subsection{Modes in a twisted module}\label{smodes}

Throughout this section, $V$ will be a vertex algebra, $\ph$ a locally finite automorphism of $V$,
and $W$ a $\ph$-twisted $V$-module. We will again write $\ph=\si e^{-2\pi\ii\N}$ with
commuting semisimple $\si\in\Aut(V)$ and locally nilpotent $\N\in\Der(V)$ (see \seref{sconlf}). 
In particular, $\bar V=V$ as $\ph$ is locally finite. 

We will denote by
\begin{equation}\label{twlog4}
V_\al = \{ a\in V \,|\, \si a = e^{-2\pi\ii\al} a \} \,, \qquad \al\in\CC/\ZZ \,,
\end{equation}
the eigenspaces of $\si$. Then by \eqref{twlog2a} we have
$Y(V_\al)\subset \LF_\al(W)$, which
means that all powers of $\z$ in $Y(a,\z)$ belong to the coset $-\al$ for $a\in V_\al$.

\begin{definition}\label{dltwlog4}
For $a\in V_\al$, $\al\in\CC/\ZZ$ and $m\in\al$, the \emph{$(m+\N)$-th mode} of $a$ is defined as
\begin{equation*}%\label{twlog9}
a_{(m+\N)} = \Res_\z \z^m X(a,\z) = \Res_\z Y(\z^{m+\N} a,\z) \in \End(W) \,,
\end{equation*}
where %as usual $\Res_\z$ denotes the coefficient of $\z^{-1}$, and
$\z^{m+\N} = \z^m e^{\ze\N}$.
\end{definition}

Since $\N(V_\al)\subset V_\al$ and $X(a,\z)$ is independent of $\ze$, we have
\begin{equation*}%\label{twlog9x}
X(a,\z) = \sum_{m\in\al} a_{(m+\N)} \z^{-m-1} \,, \qquad a\in V_\al \,.
\end{equation*}
We can recover the field $Y(a,\z)$ from the modes of $\N^l a$ $(l\ge0)$ as follows:
\begin{equation}\label{twlog10}
\begin{split}
Y(a,\z) &= X(e^{-\ze\N} a,\z) = \sum_{m\in\al} (e^{-\ze\N} a)_{(m+\N)} \z^{-m-1} \\
&= \sum_{m\in\al} (\z^{-m-1-\N} a)_{(m+\N)} \,, \qquad a\in V_\al \,.
\end{split}
\end{equation}
Note that for every $a\in V_\al$, $m\in\al$ and $v\in W$, there is an integer $L$ such that
\begin{equation*}%\label{twlog9y}
a_{(m+i+\N)} v = 0 \quad\text{for all}\quad  i\in\ZZ \,, \; i\ge L \,.
\end{equation*}

\subsection{Borcherds identity}\label{stwbor}

Now we can derive the main identity satisfied by the modes.

\begin{theorem}\label{tltwlog5}
Let\/ $V$ be a vertex algebra, $\ph$ a locally finite automorphism of\/ $V$,
and\/ $W$ a\/ $\ph$-twisted\/ $V$-module.
Then we have the {Borcherds identity}
\begin{equation}\label{twlog11}
\begin{split}
\sum_{i=0}^\infty & (-1)^i \binom{n}{i} a_{(m+n-i+\N)}(b_{(k+i+\N)}v) \\
-\sum_{i=0}^\infty & (-1)^{n+i} \binom{n}{i}  b_{(k+n-i+\N)}(a_{(m+i+\N)}v) \\ 
&= \sum_{j=0}^\infty \Bigl( \Bigl( \binom{m+\N}{j} a \Bigr)_{(n+j)}b \Bigr)_{(m+k-j+\N)}v \,,
\end{split}
\end{equation}
for\/ $a\in V_\al$, $b\in V_\be$, $v\in W$, and\/ $m\in\al$, $k\in\be$, $n\in\ZZ$.
\end{theorem}
\begin{proof}
Notice that all sums in \eqref{twlog11} are finite.
Let $N$ be such that $(\N^l a)_{(j)}b=0$ for all $l\ge0$, $j\ge N$. Then
\begin{equation}\label{twlog12a}
\z_{12}^N \, Y(a,\z_1) Y(b,\z_2) v = \z_{12}^N \, Y(b,\z_2) Y(a,\z_1) v \,,
\end{equation}
and let us denote both sides by $F(a,b;\z_1,\z_2)$.
Using the expansions \eqref{vert14} and \eqref{vert15}, we compute for $n\le N-1$:
\begin{equation}\label{twlog12}
\begin{split}
\io_{\z_1,\z_2} \z_{12}^n \, & Y(a,\z_1) Y(b,\z_2) v - \io_{\z_2,\z_1} \z_{12}^n \, Y(b,\z_2) Y(a,\z_1) v \\
&= F(a,b;\z_1,\z_2) \, (\io_{\z_1,\z_2} - \io_{\z_2,\z_1}) \z_{12}^{n-N} \\
&= F(a,b;\z_1,\z_2) \, \d_{\z_2}^{(N-1-n)} \de(\z_1,\z_2) \,.
\end{split}
\end{equation}
Let us now replace in this equation $a$ with $\z_1^{m+\N} a$ and $b$ with $\z_2^{k+\N} b$ to get
\begin{equation}\label{twlog12x}
\begin{split}
\io_{\z_1,\z_2} \z_{12}^n \, & \z_1^m \z_2^k X(a,\z_1) X(b,\z_2) v 
- \io_{\z_2,\z_1} \z_{12}^n \, \z_1^m \z_2^k X(b,\z_2) X(a,\z_1) v \\
&= F(\z_1^{m+\N} a,\z_2^{k+\N} b;\z_1,\z_2) \, \d_{\z_2}^{(N-1-n)} \de(\z_1,\z_2) \,.
\end{split}
\end{equation}

If we then take $\Res_{\z_1} \Res_{\z_2}$ of the left-hand side of \eqref{twlog12x},
we will obtain the left-hand side of \eqref{twlog11}, for any $n\in\ZZ$.
By the property \eqref{vert16} of the delta function,
if we take $\Res_{\z_1}$ of the right-hand side of \eqref{twlog12x}, we will get
\begin{align*}
\d_{\z_1}^{(N-1-n)} F(\z_1^{m+\N} a, \z_2^{k+\N} b;\z_1,\z_2) \Big|_{\z_1=\z_2} \,.
\end{align*}
In this formula, we can replace $\d_{\z_1}$ with $D_{\z_1}$, because $F(\z_1^{m+\N} a, \z_2^{k+\N} b;$ $\z_1,\z_2)$
is independent of $\ze_1$. Then using the Leibniz rule, \eqref{logf9}, \eqref{twlog3} and \eqref{twlog7}, we obtain:
\begin{align*}
\sum_{j=0}^{N-1-n} & \, D_{\z_1}^{(N-1-n-j)} F\Bigl( \binom{m+\N}{j} \z_2^{m-j+\N}  a , \z_2^{k+\N} b;\z_1,\z_2 \Bigr) \Big|_{\z_1=\z_2} \\
&= \sum_{j=0}^{N-1-n} Y\Bigl(\Bigl( \binom{m+\N}{j} \z_2^{m-j+\N}  a\Bigr)_{(n+j)} \bigl( \z_2^{k+\N} b \bigr),\z_2 \Bigr) v \\
&= \sum_{j=0}^{N-1-n} Y\Bigl(\z_2^{m+k-j+\N} \Bigl( \Bigl( \binom{m+\N}{j} a\Bigr)_{(n+j)} b \Bigr),\z_2 \Bigr) v \,.
\end{align*}
Now taking $\Res_{\z_2}$ gives exactly the right-hand side of \eqref{twlog11}.
This proves \eqref{twlog11} in the case $n\le N-1$.
When $n\ge N$, the left-hand side of \eqref{twlog12} is $0$. 
The right-hand side of \eqref{twlog11} is also obviously $0$ for $n\ge N$.
\end{proof}

For $\al\in\CC/\ZZ$, introduce the shifted delta function (cf.\ \cite{BK1}):
\begin{equation}\label{twlog13}
\begin{split}
\de_{\al+\N}(\z_1,\z_2) &=  \sum_{m\in\al} \z_1^{-m-1-\N} \z_2^{m+\N} \\
&= \z_1^{-m} \z_2^{m} \de(\z_1,\z_2) \, e^{(\ze_2-\ze_1)\N} \,, \qquad m\in\al \,,
\end{split}
\end{equation}
where $e^{(\ze_2-\ze_1)\N}$ is a linear map from $V$ to $V[\ze_1,\ze_2]$.

\begin{proposition}\label{pltwlog6}
The Borcherds identity\/ \eqref{twlog11} is equivalent to the equation
\begin{equation}\label{twlog14}
\begin{split}
\io_{\z_1,\z_2} \z_{12}^n \, & Y(a,\z_1) Y(b,\z_2) v - \io_{\z_2,\z_1} \z_{12}^n \, Y(b,\z_2) Y(a,\z_1) v \\
&= \sum_{j=0}^\infty Y \Bigl( \bigl( D_{\z_2}^{(j)} \de_{\al+\N}(\z_1,\z_2)  a \bigr)_{(n+j)} b,\z_2 \Bigr) v
\end{split}
\end{equation}
for\/ $a\in V_\al$, $b\in V$, $v\in W$ and\/ $n\in\ZZ$.
\end{proposition}
\begin{proof}
In the proof of \thref{tltwlog5} we saw that we can get the left-hand side of \eqref{twlog11} from the left-hand side of \eqref{twlog14}
if we replace $a$ with $\z_1^{m+\N} a$, $b$ with $\z_2^{k+\N} b$ and then take $\Res_{\z_1} \Res_{\z_2}$.
Conversely, we can go back by summing over all $m,k$. Similarly, the right-hand side of \eqref{twlog14}, which is equal to
\begin{equation*}
\sum_{j=0}^\infty \sum_{m\in\al} Y\Bigl(\Bigl( \binom{m+\N}{j} \z_1^{-m-1-\N} \z_2^{m-j+\N}  a\Bigr)_{(n+j)} b,\z_2 \Bigr) v \,,
\end{equation*}
corresponds to the right-hand side of \eqref{twlog11}.
\end{proof}

\begin{remark}\label{rtwlog6}
The Borcherds identity \eqref{twlog14} remains true without the assumption that $\ph$ is locally finite.
However, it requires that $a\in V_\al\subset\bar V$, so $\ph$ is locally finite on $a$.
\end{remark}

Next, we show that the
Borcherds identity can replace the locality and $n$-th product identity in the definition of a $\ph$-twisted module.

\begin{proposition}\label{pltwlog7}
Let\/ $V$ be a vertex algebra, $\ph$ a locally finite automorphism, $W$ a vector space, and\/ $Y\colon V\to\LF(W)$ a linear map satisfying the\/ $\ph$-equivariance\/
\eqref{twlog2} and the Borcherds identity\/ \eqref{twlog14}. Then\/ $W$ is a\/ $\ph$-twisted\/ $V$-module.
\end{proposition}
\begin{proof}
The proof follows by reversing the proofs of \thref{tltwlog5} and \prref{pltwlog6}.
Fix $a,b\in V$, and let $N\ge0$ be such that $(\N^l a)_{(j)} b = 0$ for all $l\ge0$, $j\ge N$. 
Then setting $n=N$ in \eqref{twlog14}, we obtain the locality \eqref{twlog12a} for all $v\in W$.

Note that \eqref{twlog3} is trivial for $n\ge N$. Suppose $n\le N-1$ and
denote both sides of \eqref{twlog12a} again by $F(a,b;\z_1,\z_2)$. By \eqref{twlog12} and \eqref{twlog14},
\begin{align*}
F(a,b;\z_1,\z_2) \, \d_{\z_2}^{(N-1-n)} \de(\z_1,\z_2)
= \sum_{j=0}^{N-1-n} Y \Bigl( \bigl( D_{\z_2}^{(j)} \de_{\al+\N}(\z_1,\z_2)  a \bigr)_{(n+j)} b,\z_2 \Bigr) v \,.
\end{align*}
Now if we replace $a$ with $\z_1^{m+\N} a$, where $a\in V_\al$, $m\in\al$, we will have only integral powers of $\z_1$ and no dependence on $\ze_1$. Then take $\Res_{\z_1}$ to obtain
\begin{align*}
\d_{\z_1}^{(N-1-n)} \, F(\z_1^{m+\N} a,b;\z_1,\z_2) \Big|_{\z_1=\z_2}
= \sum_{j=0}^{N-1-n}  Y \Bigl( \bigl( D_{\z_2}^{(j)} (\z_2^{m+\N}) a \bigr)_{(n+j)} b,\z_2 \Bigr) v \,.
\end{align*}
%For $n=N-1$, this implies $F(a,b;\z_2,\z_2) = Y(a_{(N-1)}b, \z_2)$. For all $n\le N-1$, 
%We can replace $\d_{\z_1}$ with $D_{\z_1}$ in the left-hand side of the above equation and use the Leibniz rule to get
By the Leibniz rule, the left-hand side is equal to
\begin{align*}
\sum_{j=0}^{N-1-n}  D_{\z_1}^{(N-1-n-j)} \, F\bigl( D_{\z_2}^{(j)} (\z_2^{m+\N}) a,b;\z_1,\z_2\bigr) \Big|_{\z_1=\z_2} \,.
\end{align*}
Then by induction on $N-1-n$, it follows that
\begin{align*}
D_{\z_1}^{(N-1-n)} F(a,b;\z_1,\z_2) \Big|_{\z_1=\z_2} = Y (a_{(n)} b,\z_2) \,,
\end{align*}
which is exactly \eqref{twlog3}.
\end{proof}

\subsection{Commutator formulas}\label{scomf}

Setting $n=0$ in the Borcherds identity~\eqref{twlog11}, we obtain the \emph{commutator formula}
\begin{equation}\label{twlog15}
\bigl[ a_{(m+\N)}, b_{(k+\N)} \bigr] = \sum_{j=0}^\infty \Bigl( \Bigl( \binom{m+\N}{j} a \Bigr)_{(j)}b \Bigr)_{(m+k-j+\N)} \,,
\end{equation}
where $a\in V_\al$, $b\in V_\be$, $m\in\al$, $k\in\be$. Similarly, from \eqref{twlog14} we have:
\begin{equation}\label{twlog16}
\bigl[ Y(a,\z_1), Y(b,\z_2) \bigr] 
= \sum_{j=0}^\infty Y \Bigl( \bigl( D_{\z_2}^{(j)} \de_{\al+\N}(\z_1,\z_2)  a \bigr)_{(j)} b,\z_2 \Bigr)
\end{equation}
for all $a\in V_\al$ and $b\in V$.
Extracting the coefficient of $\z_1^{-m-1-\N} a$, we deduce another useful formula:
\begin{equation}\label{twlog17}
\bigl[ a_{(m+\N)}, Y(b,\z) \bigr] 
= \sum_{j=0}^\infty Y \Bigl( \bigl( \binom{m+\N}{j} \z^{m-j+\N}  a \bigr)_{(j)} b,\z \Bigr) \,.
\end{equation}
%
%\begin{remark}\label{rtwlog7}
As in \reref{rtwlog6}, equations \eqref{twlog16} and \eqref{twlog17} hold
without the assumption that $\ph$ is locally finite, but they require $a\in V_\al \subset\bar V$.
%\end{remark}

From \eqref{twlog16} we can derive a formula for the propagator $P( a, b; \z_1,\z_2)$
of $Y(a,\z_1)$ and $Y(b,\z_2)$ (see \seref{sloccol}). Recall that $\al_0\in\al$ is such that $-1<\Re\al_0\le0$.

\begin{lemma}\label{lltwlog8}
For any\/ $a\in V_\al$ and\/ $b\in V$, we have
\begin{equation*}%\label{twlog23}
\z_{12}^N \, P( a, b; \z_1,\z_2)
= \sum_{j=0}^{N-1} \sum_{i=0}^j \z_{12}^{N-1-i} \,
Y\Bigl( \bigl( D_{\z_2}^{(j-i)} \z_1^{-\al_0-\N} \z_2^{\al_0+\N}  a \bigr)_{(j)} b,\z_2 \Bigr) \,,
%&= \sum_{j=0}^{N-1} \sum_{i=0}^j \binom{N}{i} Y\Bigl( D_{\z_2}^{(j-i)} \bigl( \z_{12}^{N-1-i} \z_1^{-\al_0-\N} %\z_2^{\al_0+\N}  a \bigr)_{(j)} b,\z_2 \Bigr) \,,
\end{equation*}
where\/ $N$ is such that\/ $(\N^l a)_{(j)}b=0$ for all\/ $l\ge0$, $j\ge N$.
\end{lemma}
\begin{proof}
By comparing the powers of $\z_1$ in \eqref{twlog16}, we obtain
\begin{equation*}%\label{twlog24}
P( a, b; \z_1,\z_2) = 
\sum_{j=0}^\infty Y \Bigl( \bigl( D_{\z_2}^{(j)} \io_{\z_1,\z_2} \z_{12}^{-1} \z_1^{-\al_0-\N} \z_2^{\al_0+\N}  a \bigr)_{(j)} b,\z_2 \Bigr) \,,
\end{equation*}
using \eqref{vert14}, \eqref{vert15} and \eqref{twlog13}. In this equation, the sum over $j$ goes only up to $j=N-1$.
Then we apply the Leibniz rule and multiply by $\z_{12}^N$ to finish the proof.
\end{proof}

The next result is useful for constructing $\ph$-twisted modules.

\begin{proposition}\label{pltwlog9}
Let\/ $V$ be a vertex algebra, $\ph$ an automorphism, $W$ a vector space, and\/ $Y\colon V\to\LF(W)$ a linear map satisfying the 
$\ph$-equivariance\/ \eqref{twlog2} and the
commutator formula\/ \eqref{twlog16} for\/ $a\in V_\al$, $b\in V$. Then the logarithmic fields\/ 
$Y(a,\z)$ and\/ $Y(b,\z)$ are local, and the\/ $n$-th product identity\/
\eqref{twlog3} holds for\/ $a\in V_\al$, $b\in V$ and all\/ $n\ge0$.
\end{proposition}
\begin{proof}
As before, let $N\ge0$ be such that $(\N^l a)_{(j)} b = 0$ for all $l\ge0$, $j\ge N$. 
Then in \eqref{twlog16} the sum over $j$ goes only up to $j=N-1$.
Using $\z_{12} \,\de_{\al+\N}(\z_1,\z_2) = 0$, we derive from \eqref{twlog16} the locality \eqref{logf7} of $a(\z)=Y(a,\z)$ and $b(\z)=Y(b,\z)$.  

By definition, their $n$-th product is $0$ for $n\ge N$.
To find it for $0\le n\le N-1$, we apply \prref{ptlogf5} and \leref{lltwlog8}. 
Let us use the convention that $x^{(k)}=0$ for $k<0$. Then
\begin{align*}
Y&(a,\z)_{(n)} Y(b,\z) \\
&= D_{\z_1}^{(N-1-n)} \sum_{i,j=0}^{N-1} \z_{12}^{N-1-i} \,
Y\Bigl( \bigl( D_{\z_2}^{(j-i)} \z_1^{-\al_0-\N} \z_2^{\al_0+\N}  a \bigr)_{(j)} b,\z_2 \Bigr) \Big|_{ \z_1=\z_2=\z} \,.
\end{align*}
For a fixed $j$, we calculate using the Leibniz rule:
\begin{align*}
\sum_{i=0}^{N-1} & D_{\z_1}^{(N-1-n)} \bigl( \z_{12}^{N-1-i}
 \bigl( D_{\z_2}^{(j-i)} \z_1^{-\al_0-\N} \z_2^{\al_0+\N}  \bigr) \bigr) \big|_{ \z_1=\z_2=\z} \\
&= \sum_{i=0}^{N-1} D_{\z_1}^{(i-n)} D_{\z_2}^{(j-i)} 
\bigl( \z_1^{-\al_0-\N} \z_2^{\al_0+\N}  \bigr) \big|_{ \z_1=\z_2=\z} \\
&= D_{\z}^{(j-n)} \bigl( \z^{-\al_0-\N} \z^{\al_0+\N}  \bigr)
= \de_{j,n} \,.
\end{align*}
Therefore, $Y(a,\z)_{(n)} Y(b,\z) = Y(a_{(n)}b,\z)$.
\end{proof}

As another application of \leref{lltwlog8}, we obtain a formula relating the $(-1)$-st product with the normally ordered product
given by \deref{dlogf4} (cf.\ \cite[(3.13)]{BK1}). 

\begin{lemma}\label{lltwlog11}
In every\/ $\ph$-twisted\/ $V$-module, we have 
\begin{equation*}%\label{twlog25}
{:} Y(a,\z) Y(b,\z) {:} = \sum_{j=-1}^{N-1} \z^{-j-1} \, Y \Bigl( \bigl( \binom{\al_0+\N}{j+1} a \bigr)_{(j)} b,\z \Bigr)
\end{equation*}
for\/ $a\in V_\al$ and\/ $b\in V$.
\end{lemma}
\begin{proof}
We proceed as in the proof of \prref{pltwlog9} for $n=-1$. We calculate for a fixed $0\le j\le N-1$:
\begin{align*}
\sum_{i=0}^{N-1} & D_{\z_1}^{(N)} \bigl( \z_{12}^{N-1-i}
 \bigl( D_{\z_2}^{(j-i)} \z_1^{-\al_0-\N} \z_2^{\al_0+\N}  \bigr) \bigr) \big|_{ \z_1=\z_2=\z} \\
&= \sum_{i=0}^{j} D_{\z_1}^{(i+1)} D_{\z_2}^{(j-i)} 
\bigl( \z_1^{-\al_0-\N} \z_2^{\al_0+\N}  \bigr) \big|_{ \z_1=\z_2=\z} \\
&= D_{\z}^{(j+1)} \bigl( \z^{-\al_0-\N} \z^{\al_0+\N}  \bigr) 
- D_{\z_1}^{(0)} D_{\z_2}^{(j+1)} 
\bigl( \z_1^{-\al_0-\N} \z_2^{\al_0+\N}  \bigr) \big|_{ \z_1=\z_2=\z} \\
&= - \binom{\al_0+\N}{j+1} \z^{-j-1} \,.
\end{align*}
The rest of the proof follows again from \prref{ptlogf5} and \leref{lltwlog8}. 
\end{proof}

\subsection{Action of the Virasoro algebra}\label{svirac}

In this subsection, we assume that the vertex algebra $V$ is \emph{conformal}, 
i.e., there exist $\om\in V$ (conformal vector) and $c\in\CC$ (central charge) such that (see, e.g., \cite{K2}):
\begin{equation*}%\label{twlog18}
\om_{(0)} = T \,, \quad \om_{(1)} \om = 2\om \,, \quad \om_{(2)} \om = \frac{c}2 \vac \,, 
\end{equation*}
and $\om_{(j)} \om = 0$ for $j\ge 3$. Then the modes $L_n = \om_{(n+1)}$ give a representation of the
Virasoro Lie algebra on $V$ with a central charge $c$. In addition, it is usually assumed that the operator $L_0$ is semisimple
on~$V$. 

Consider a (not necessarily locally finite) $\ph\in\Aut(V)$ such that $\ph(\om)=\om$, and a $\ph$-twisted $V$-module $W$.
Then the field $Y(\om,\z)$ on $W$ has only integral powers of $\z$ and no $\ze$, and its modes
\begin{equation}\label{twlog19}
Y(\om,\z) = \sum_{n\in\ZZ} L_n^W \z^{-n-2} \,, \qquad L_n^W \in\End(W) \,,
\end{equation}
give a representation of the Virasoro algebra on $W$ with the same central charge $c$.

Applying the commutator formula \eqref{twlog17}, we obtain:
\begin{align*}%\label{twlog20}
[L_{-1}^W, Y(a,\z) ] &= Y(Ta,\z) = D_\z Y(a,\z) \,, \\
%\label{twlog21}
[L_{0}^W, Y(a,\z) ] &= \z Y(Ta,\z) + Y(L_0 a,\z) \,, \qquad a\in V\,.
\end{align*}
Then from $\z D_\z=D_\ze$, we have
\begin{equation*}%\label{twlog23}
[L_{0}^W, Y(a,\z) ] = (D_\ze+\De) Y(a,\z) \,, \quad\text{if}\quad L_0 a=\De a \,.
\end{equation*}

\begin{remark}\label{rtwlog8}
Assume that $L_0$ is semisimple on $V$, but $\ph$ is not semisimple. 
Then $L_{0}^W$ is not semisimple on $W$.
Thus, $W$ is a (untwisted) $V^\ph$-module with a non-semisimple action of $L_{0}^W$, also known as
a \emph{logarithmic module} (see \cite{AM}).
\end{remark}

\begin{lemma}\label{lltwlog10}
Let\/ $W$ be a\/ $\ph$-twisted\/ $V$-module.
Assume that the operator\/ $L_0$ is semisimple on\/ $V$ with integral eigenvalues,
and the operator\/ $e^{2\pi\ii L_{0}^W}$ is well defined on\/ $W$.
Then
\begin{equation*}%\label{twlog22}
e^{2\pi\ii L_{0}^W} Y(a,\z) e^{-2\pi\ii L_{0}^W} = e^{2\pi\ii D_\ze} Y(a,\z) = Y(\ph a,\z)
\end{equation*}
when acting on\/ $W$, for every\/ $a\in V$.
\end{lemma}
\begin{proof}
Indeed,
\begin{align*}
e^{2\pi\ii L_{0}^W} Y(a,\z) e^{-2\pi\ii L_{0}^W} = e^{2\pi\ii \ad(L_{0}^W) } Y(a,\z)
= e^{ 2\pi\ii (D_\ze+\De) } Y(a,\z) 
\end{align*}
when $L_0 a=\De a$.
\end{proof}

\leref{lltwlog10} can be used to define $e^{2\pi\ii L_{0}^W}$ on the whole $W$, provided it can be defined on a set of generators of
$W$ as a $\ph$-twisted $V$-module; in particular, when these generators are eigenvectors of $L_{0}^W$. 
This can be used to define a grading of $W$ as in \cite[Definition 3.1]{H}.

\section{Twisted modules of affine and Heisenberg vertex algebras}\label{s6}

In this section, we describe all twisted modules of affine and Heisenberg vertex algebras in terms of modules over
certain twisted versions of the corresponding Lie algebras. We also determine the action of the Virasoro algebra. 
For the Heisenberg vertex algebra, all twisted irreducible highest-weight modules are constructed explicitly.

\subsection{Universal affine vertex algebras}\label{aff}

Let us first recall the definition of affine Lie algebras, following \cite{K1}.
Consider a finite-dimensional Lie algebra $\lieg$ equipped with a nondegenerate symmetric invariant bilinear form $(\cdot|\cdot)$,
normalized so that the square length of a long root is $2$ in the case when $\lieg$ is simple. 
The \emph{affine Lie algebra}
$\hat\lieg = \lieg[t,t^{-1}] \oplus \C K$
has the Lie brackets
\begin{equation}\label{aff1}
[at^m,bt^n] = [a,b]t^{m+n} + m \delta_{m,-n} (a|b) K \,, \qquad [K,at^m]=0 \,.
\end{equation}
For a fixed $\ka\in\CC$, called the \emph{level}, the (generalized) \emph{Verma module} 
$M(\ka\Lambda_0) = \Ind^{\hat\lieg}_{\lieg[t]\oplus\CC K} \CC$
is defined by letting $\lieg[t]$ act trivially on $\CC$
and $K$ act as $\ka$. Then $M(\ka\Lambda_0)$ is a highest-weight $\hat\lieg$-module
with a highest-weight vector the image of $1\in\C$, which will be denoted $\vac$. 
Notice that as a vector space, $M(\ka\Lambda_0) \cong U(\lieg[t^{-1}]t^{-1})$.

Due to \cite{FZ}, $M(\ka\Lambda_0)$ has the structure of a vertex algebra,
which is called the \emph{universal affine vertex algebra} at level $\ka$ and is denoted $V^\ka(\lieg)$.
It has a vacuum vector $\vac$ and is generated by the local fields 
\begin{equation*}%\label{aff2}
Y((at^{-1})\vac,\z) = \sum_{m\in\ZZ} (at^m) \z^{-m-1} \,, \qquad a\in\lieg
\end{equation*}
(see, e.g., \cite{K2} for more details).
For simplicity of notation, let us identify $a\in\lieg$ with $(at^{-1})\vac\in V^\ka(\lieg)$; 
then $a_{(m)} = at^m$ as operators on $V^\ka(\lieg)$.
By the commutator formula \eqref{vert11}, the Lie brackets \eqref{aff1} are equivalent to the relations
\begin{equation}\label{aff3}
a_{(0)} b = [a,b] \,, \quad a_{(1)} b = (a|b) \ka\vac \,, \quad a_{(j)} b = 0 \quad (j\ge2)
\end{equation}
for $a,b\in\lieg$.

Now suppose that $\lieg$ is simple or abelian, and let $h^\vee$ be the dual Coxeter number of $\lieg$ in the case when it is simple.
When $\lieg$ is abelian, we set $h^\vee=0$. Then the vertex algebra $V^\ka(\lieg)$ is conformal for $\ka\neq -h^\vee$.
Pick dual bases $\{v_i\}$ and $\{v^i\}$ for $\lieg$ with respect to $(\cdot|\cdot)$. The conformal vector $\om\in V^\ka(\lieg)$ is given by the \emph{Sugawara construction}
\begin{equation}\label{aff4}
\om = \frac1{2(\ka+h^\vee)} \sum_{i=1}^{\dim\lieg} v^i_{(-1)} v_i \,, \qquad \ka\neq -h^\vee
\end{equation}
(see, e.g., \cite{K2}). The Virasoro central charge is $c=\ka\dim\lieg / (\ka+h^\vee)$.
The operator $L_0$ satisfies $L_0\vac=0$, $L_0 a=a$ $(a\in\lieg)$, and it defines a $\ZZ_+$-grading of $V^\ka(\lieg)$ by its eigenvalues. 

\subsection{$\ph$-twisted modules of $V^\ka(\lieg)$}\label{sphtwv}

From now on, $\ph$ will be an automorphism of $\lieg$ such that $(\cdot|\cdot)$ is $\ph$-invariant.
Then $\ph$ induces automorphisms of $\hat\lieg$ and $V=V^\ka(\lieg)$, which we will again denote as $\ph$.
Notice that $\ph(\om)=\om$. Since the eigenspaces of $L_0$ in $V$ are finite-dimensional and
$\ph$-invariant, $\ph$ is locally finite on $V$. 

Writing again $\ph=\si e^{-2\pi\ii\N}$, we have $\si\in\Aut(\lieg)$,
$\N\in\Der(\lieg)$, and
\begin{equation}\label{aff5}
(\si a| \si b) = (a|b) \,, \qquad (\N a|b) + (a|\N b) = 0 \,.
\end{equation}
As before, we denote the eigenspaces of $\si$ by
\begin{equation*}%\label{aff6}
\lieg_\al=\{a\in\lieg \, | \,\si a = e^{-2\pi\ii\al} a\} \,, \qquad \al\in\CC/\ZZ \,.
\end{equation*}
If $W$ is a $\ph$-twisted $V$-module, then by \eqref{twlog15} and \eqref{aff3}, we have:
\begin{equation}\label{aff7}
\bigl[ a_{(m+\N)}, b_{(k+\N)} \bigr] = [a,b]_{(m+k+\N)} + \de_{m,-n} ((m+\N)a|b) \ka I %\,, \qquad a,b\in\lieg \,.
\end{equation}
for $a\in\lieg_\al$, $b\in\lieg_\be$, $m\in\al$, $k\in\be$. Hence, the modes $a_{(m+\N)}$ close a Lie algebra, which can be described as follows (cf.\ \cite[Chapter 8]{K1}).

Let $\tilde\lieg = \bigoplus_{\al\in\CC/\ZZ} \lieg[t]t^\al$ be the \emph{loop algebra}, whose elements are finite sums of $at^m$ ($a\in\lieg$, $m\in\CC$), with the Lie bracket $[at^m,bt^n] = [a,b]t^{m+n}$. We define an automorphism $\tilde\si$ of $\tilde\lieg$ by $\tilde\si(at^m)=e^{2\pi\ii m}\si(a)t^m$. The subalgebra $\tilde\lieg_\si$ of fixed points under $\tilde\si$ is spanned by $at^m$ ($a\in\lieg_\al$, $m\in\al$).
The loop algebra $\tilde\lieg$ has a $2$-cocycle $\ga$ given by
\begin{equation*}%\label{aff8}
\ga(at^m,bt^n) = m\de_{m,-n} (a|b) = \Res_t (\d_t(at^m) | bt^n) \,,
\end{equation*}
which gives rise to a central extension of $\tilde\lieg$ similar to $\hat\lieg$ (see \eqref{aff1}). 
When restricted to $\tilde\lieg_\si$, we obtain
the Lie algebra $\hat\lieg_\si=\tilde\lieg_\si\oplus\CC K$. It is easy to check that
\begin{equation*}%\label{aff9}
\ga_\N(at^m,bt^n) = \de_{m,-n} ((m+\N)a|b) = \Res_t \bigl( (\d_t+t^{-1}\N)(at^m) \big| bt^n \bigr)
\end{equation*}
again defines a $2$-cocycle on $\tilde\lieg$. If we use $\ga_\N$ instead of $\ga$,
we obtain the Lie algebra $\hat\lieg_\ph=\tilde\lieg_\si\oplus\CC K$.

\begin{definition}\label{daff1}
The \emph{$\ph$-twisted affinization} of $\lieg$ is the Lie algebra $\hat\lieg_\ph$ spanned by a central element $K$ and elements $at^m$ ($a\in\lieg_\al$, $m\in\al$), with the
Lie bracket
\begin{equation}\label{aff10}
[at^m,bt^n] = [a,b]t^{m+n} + \de_{m,-n} ((m+\N)a|b) K \,.
\end{equation}
\end{definition}

\begin{proposition}\label{paff3}
When the Lie algebra\/ $\lieg$ is simple, there exists a Dynkin diagram automorphism\/ $\mu$ of\/ $\lieg$, such that\/
$\hat\lieg_\ph\cong\hat\lieg_\mu$ is a (possibly twisted) affine Kac--Moody algebra.
\end{proposition}
\begin{proof}
We have $\N=\ad_y$ for some $y\in\lieg$. Then
\begin{equation*}
\Res_t ( t^{-1}\N(at^m) | bt^n ) = \Res_t ( y t^{-1} | [at^m,bt^n] ) \,,
\end{equation*}
and the cocycle $\ga_\N$ is equivalent to $\ga$. Hence, $\hat\lieg_\ph\cong\hat\lieg_\si$.

We can write $\si = \mu e^{2\pi\ii\ad_x}$ for some semisimple $x\in\lieg$ and a Dynkin diagram automorphism $\mu$,
so that $\mu$ commutes with $\ad_x$. Then the map $t^{\ad_x}$, defined by $t^{\ad_x}(at^m) = at^{m+p}$ whenever $[x,a]=pa$, is an isomorphism from $\tilde\lieg_\si$ to $\tilde\lieg_\mu$ (cf.\ \cite[Proposition 8.5]{K1}). It lifts to an isomorphism $\hat\lieg_\si\cong\hat\lieg_\mu$, since the cocycle
\begin{equation*}
\Res_t \bigl( \d_t( t^{\ad_x} at^m) \big| t^{\ad_x} bt^n \bigr) = \Res_t \bigl( (\d_t + t^{-1} \ad_x) (at^m) \big| bt^n \bigr) 
= \ga_{\ad_x} (at^m | bt^n )
\end{equation*}
is equivalent to $\ga$.
Finally,  $\hat\lieg_\mu$ is an affine Kac--Moody algebra by \cite[Theorem 8.3]{K1}.
\end{proof}

A $\hat\lieg_\ph$-module $W$ is called \emph{restricted} if for every $a\in\lieg_\al$, $m\in\al$, $v\in W$, there is an integer $L$ such that $(at^{m+i}) v = 0$ for all $i\in\ZZ$, $i\ge L$. For example, every highest-weight $\hat\lieg_\ph$-module is restricted (see \cite{K1}).
We say that $W$ has \emph{level $\ka$} if $K$ acts on it as $\ka I$.
Then we have the following correspondence of modules (cf.\ \cite{Li2,KRR}).

\begin{theorem}\label{taff2}
Every\/ $\ph$-twisted\/ $V^\ka(\lieg)$-module is a restricted\/ $\hat\lieg_\ph$-module of level\/ $\ka$ and, conversely,
every restricted\/ $\hat\lieg_\ph$-module of level\/ $\ka$ uniquely extends to a\/ $\ph$-twisted\/ $V^\ka(\lieg)$-module.
\end{theorem}
\begin{proof}
In one direction the statement is obvious from the definitions. Conversely, suppose that $W$ is a restricted $\hat\lieg_\ph$-module of level $\ka$. For $a\in\lieg_\al$, we define the logarithmic field $Y(a,\z)\in\LF_\al(W)$ by
\begin{equation}\label{aff11}
Y(a,\z) = \sum_{m\in\al} \z^{-m-1} \bigl( (e^{-\ze\N}a)t^m \bigr)   \,.
\end{equation}
Then \eqref{twlog2a} holds, which implies the $\ph$-equivariance \eqref{twlog2}.
By \eqref{twlog10}, the modes of $a$ are $a_{(m+\N)}=at^m$.

Comparing \eqref{aff7} and \eqref{aff10}, we see that the commutator formula \eqref{twlog16}
holds for $a\in\lieg_\al$, $b\in\lieg_\be$. By \prref{pltwlog9}, the fields $Y(a,\z)$ are local and satisfy the $n$-th product identity
\eqref{twlog3} for $n\ge0$. 
Let $\W$ be the local collection $\{Y(a,\z)\}_{a\in\lieg}$, and $\bar\W\subset\LF(W)$ be the vertex algebra generated by it
(see \thref{tlogf5}). Since $V^\ka(\lieg) \cong U(\lieg[t^{-1}]t^{-1})$, the map $Y$ can be extended uniquely to a vertex algebra
homomorphism from $V^\ka(\lieg)$ to $\bar\W$. This endows $W$ with the structure of a $\ph$-twisted $V^\ka(\lieg)$-module.
\end{proof}

As in \seref{svirac}, the modes of $Y(\om,\z)$ give a representation of the Virasoro Lie algebra on 
every $\ph$-twisted $V^\ka(\lieg)$-module $W$. To state the explicit formula,
let us define a linear operator $\S$ on $\lieg$ by $\S a= \al_0\, a$ for $a\in\lieg_\al$, $\al\in\CC/\ZZ$ and
$\al_0\in\al$ such that $-1<\Re\al_0\le0$. Recall that the normally ordered product of two logarithmic fields is given by \deref{dlogf4}.

\begin{lemma}\label{laff5}
In every\/ $\ph$-twisted\/ $V^\ka(\lieg)$-module\/ $W$, we have
\begin{equation*}%\label{aff13b}
2(\ka+h^\vee) Y(\om,\z) = \sum_{i=1}^{\dim\lieg} {:} X(v^i,\z) X(v_i,\z) {:} 
-\z^{-1} X( \bar\om, \z ) 
-\z^{-2} \ka\tr \binom{\S}2 I \,,
\end{equation*}
where
\begin{equation*}%\label{aff13c}
\bar\om = \sum_{i=1}^{\dim\lieg} [(\S+\N)v^i, v_i] \,,
\end{equation*}
using the notation from\/ \eqref{aff4}.
\end{lemma}
\begin{proof}
Applying \leref{lltwlog11} and \eqref{aff3}, we obtain for $a,b\in\lieg$:
\begin{equation*}%\label{aff13}
\begin{split}
Y(a_{(-1)}b,\z) &= {:} Y(a,\z) Y(b,\z) {:} \\
&-\z^{-1} Y\bigl( \bigl[(\S+\N)a, b\bigr], \z \bigr) 
-\z^{-2} \Bigl( \binom{\S+\N}2 a \Big| b \Bigr) \ka I \,.
\end{split}
\end{equation*}
Then we use this with \eqref{aff4} to find $Y(\om,\z)$. 
Note that $Y(\om,\z)$ is independent of $\ze$, because $\ph\om=\om$.
Hence, we can set $\ze=0$ and replace $Y$ with $X$ (see \leref{lltwlog3}). 
Finally, 
\begin{equation*}
\sum_{i=1}^{\dim\lieg} \Bigl( \binom{\S+\N}2 v^i \Big| v_i \Bigr)
= \tr\binom{\S+\N}2 = \tr\binom{\S}2 \,,
\end{equation*}
since we can find a basis for $\lieg$ in which $\S$ is diagonal and $\N$ is strictly upper triangular.
\end{proof}

Note that we can pick the dual bases for $\lieg$ so that $v^i\in\lieg_{\al^i}$ and $v_i\in\lieg_{-\al^i}$ for some $\al^i\in\CC/\ZZ$.
Moreover, $\bar\om\in\lieg_0$ since $\ph\bar\om=\bar\om$. Then from \leref{laff5}, we obtain for the modes \eqref{twlog19}:
\begin{equation}\label{aff13o}
2(\ka+h^\vee) L_n^W
= \sum_{i=1}^{\dim\lieg} \sum_{m\in\al^i} {:} (v^i t^m) (v_i t^{n-m}) {:}
- \bar\om t^n - \de_{n,0} \ka\tr \binom{\S}2 I \,,
\end{equation}
for any $n\in\ZZ$.

\subsection{$\ph$-twisted modules of the Heisenberg vertex algebra}\label{stwheis}

Now assume that $\lieg$ is abelian, and denote it by $\lieh$ instead of $\lieg$.
The affine Lie algebra $\hat\lieh$ is called the \emph{Heisenberg Lie algebra} and
its irreducible highest-weight module $\F = M(\Lambda_0) = V^1(\lieh)$
is known as the (bosonic) \emph{Fock space} or the \emph{Heisenberg vertex algebra}.
Explicitly, the Lie bracket in $\hat\lieh$ is given by
\begin{equation}\label{aff10h}
[at^m,bt^n] = \de_{m,-n} ((m+\N)a|b) K \,.
\end{equation}
Note that $V^\ka(\lieh) \cong V^1(\lieh)$ for any $\ka\ne0$, so we can assume $\ka=1$ without loss of generality.

Let us split $\CC$ as a disjoint union of subsets $\CC^+$, $\CC^-=-\CC^+$ and $\{0\}$. We will take
\begin{equation*}%\label{aff16}
\CC^+ = \{ \ga\in\CC \,|\, \Re\ga>0 \} \cup \{ \ga\in\CC \,|\, \Re\ga=0, \, \Im\ga>0  \} \,.
\end{equation*}
Then the $\ph$-twisted affinization $\hhp$ has a triangular decomposition
$\hhp=\hhp^- \oplus \hhp^0 \oplus \hhp^+$ (direct sum of vector spaces), where
\begin{equation*}%\label{aff17}
\hhp^\pm = \Span\{ at^m \,|\, a\in\lieh_\al, \, \al\in\CC/\ZZ, \, m\in\al\cap\CC^\pm\}
\end{equation*}
and
\begin{equation*}%\label{aff18}
\hhp^0 = \Span\{ at^0 \,|\, a\in\lieh_0\} \oplus\CC K\,.
\end{equation*}
It is clear from \eqref{aff10h} that $\hhp^\pm$ are abelian subalgebras of $\hhp$, and $\hhp^0$ is a finite-dimensional subalgebra satisfying $[\hhp^0,\hhp^0] \subset \CC K$, $[\hhp^0,\hhp^\pm] =\{0\}$.

Let $W$ be an $\hhp$-module. A \emph{highest-weight vector} (also called a \emph{vacuum vector}) in $W$ is $v\in W$ such that $\hhp^+ v = 0$. All such vectors form an $\hhp^0$-submodule $R$ of $W$. If $W$ is generated by $R$ as an $\hhp$-module, we say that $W$ is a highest-weight module. As usual, examples can be constructed as induced modules. Starting from any $\hhp^0$-module $R$ such that $K=I$, we define the
(generalized) \emph{Verma module}
\begin{equation*}%\label{aff19}
M_\ph(R) = \Ind^{\hhp}_{\hhp^+\oplus\hhp^0} R \cong S(\hhp^-) \otimes_\CC R \,,
\end{equation*}
where $\hhp^+$ acts trivially on $R$.
It is a standard fact that the $\hhp$-module $M_\ph(R)$ is irreducible for any irreducible $\hhp^0$-module $R$ (cf.\ \cite{FLM,KRR}).
Therefore, all irreducible highest-weight $\hhp$-modules have this form. In addition, all of them are restricted, so they give rise to $\ph$-twisted $\F$-modules.

We will present two explicit examples of 
%a vector space $\lieh$ equipped with a nondegenerate symmetric bilinear form $(\cdot|\cdot)$ and commuting 
linear operators $\si$ and $\N$ on $\lieh$ satisfying \eqref{aff5}.
%such that $\si$ is semisimple and $\N$ is nilpotent.
Fix a positive integer $\ell$ and $\al_0\in\CC$ such that $-1<\Re\al_0\le0$, and set
$\la=e^{-2\pi\ii\al_0}$. %and $\al=\al_0+\ZZ\in\CC/\ZZ$. 

\begin{example}[$\dim\lieh=2\ell$]\label{eaff5}
Consider a vector space $\lieh$ with a basis $\{v_1,\dots,v_{2\ell}\}$ such that $(v_i|v_j)=\de_{i+j,2\ell+1}$
and %define $\si$ and $\N$ by
\begin{equation*}
\si v_i = \begin{cases} \la v_i, \quad\; 1\le i\le\ell \,, \\ 
\la^{-1} v_i, \; \ell+1\le i\le 2\ell \,,
\end{cases} \quad
\N v_i = \begin{cases} v_{i+1}, \;\;\;\, 1\le i\le\ell-1 \,, \\ 
-v_{i+1}, \; \ell+1\le i\le 2\ell-1 \,, \\
0, \qquad\;\, i=\ell, \, 2\ell \,.
\end{cases} 
\end{equation*}
Due to the symmetry $v_i \mapsto (-1)^i v_{\ell+i}$, $v_{\ell+i} \mapsto (-1)^{i+\ell+1} v_i$ $(1\le i\le\ell)$, we can assume that $\al_0\in\CC^-\cup\{0\}$.
\end{example}

\begin{example}[$\dim\lieh=2\ell-1$]\label{eaff6}
Here $\la=\pm1$, so $\al_0=0$ or $-1/2$.
Define $\lieh$ as a vector space with a basis $\{v_1,\dots,v_{2\ell-1}\}$ such that $(v_i|v_j)=\de_{i+j,2\ell}$
and %define $\si$ and $\N$ by
\begin{equation*}
\si v_i = \la v_i, \;\;\; 1\le i\le2\ell-1 \,, \qquad
\N v_i = \begin{cases} (-1)^{i+1} v_{i+1}, \; 1\le i\le 2\ell-2 \,, \\
0, \qquad\qquad\quad i=2\ell-1 \,.
\end{cases} 
\end{equation*}
\end{example}

\begin{proposition}\label{paff3}
Let\/ $\lieh$ be a finite-dimensional vector space, equipped with a nondegenerate symmetric bilinear form\/ $(\cdot|\cdot)$
and with commuting linear operators\/ $\si$, $\N$ satisfying \eqref{aff5}, such that\/ $\si$ is invertible and semisimple and $\N$ is 
nilpotent. Then\/ $\lieh$ is an orthogonal direct sum of subspaces that are like Examples \ref{eaff5} and \ref{eaff6}.
\end{proposition}
\begin{proof}
This follows from the well-known classification, up to conjugation, of orthogonal and skew-symmetric matrices over $\CC$
(see \cite{Ga,HM}).
\end{proof}

In the next two subsections, we will consider separately the above two examples. In each case, we will describe explicitly the $\ph$-twisted affinization $\hat\lieh_\ph$ and its irreducible highest-weight modules $M_\ph(R)$. We will also determine the action of the Virasoro algebra using \eqref{aff13o}.
Note that in \eqref{aff13o}, we have $\bar\om=0$ and the normally ordered product is needed only for $n=0$,
since $\lieh$ is abelian.

\subsection{The case $\dim\lieh=2\ell$}\label{stwheis1}

First, let $\lieh$ be as in \exref{eaff5}. Then $\hat\lieh_\ph$ is the Lie algebra spanned by a central element $K$ and elements $v_i t^{\al_0+n}$, $v_{\ell+i} t^{-\al_0+n}$ ($1\le i\le\ell$, $n\in\ZZ$), with Lie brackets given by \eqref{aff10h}. 
More explicitly, for $1\le i\le\ell$ and $1\le j\le 2\ell$, we have:
\begin{align*}%\label{aff14}
[v_i t^m, v_j t^k] &= m \de_{m+k,0}\de_{i+j,2\ell+1} K + \de_{m+k,0} (1-\de_{i,\ell}) \de_{i+j,2\ell} K \,, \\
[v_{\ell+i} t^m, v_j t^k] &= m \de_{m+k,0}\de_{i+j,\ell+1} K - \de_{m+k,0} (1-\de_{i,\ell}) \de_{i+j,\ell} K \,.
\end{align*}
It follows from \eqref{twlog10} that for $1\le j\le\ell$:
\begin{align*}%\label{aff14y}
Y(v_j,\z) &= \sum_{i=j}^\ell \sum_{m\in\al_0+\ZZ} (-\ze)^{(i-j)} (v_i t^m) \z^{-m-1} \,, \\
Y(v_{\ell+j},\z) &= \sum_{i=j}^\ell \sum_{m\in-\al_0+\ZZ} \ze^{(i-j)} (v_{\ell+i} t^m) \z^{-m-1} \,.
\end{align*}
The action of the Virasoro algebra is determined by \eqref{aff13o} with $\bar\om=0$.
The dual basis $\{v^i\}$ to the basis $\{v_i\}$ is given by $v^i=v_{2\ell+1-i}$.
As we already pointed out, the normally ordered product in \eqref{aff13o} is needed only for $L_0^W$.
Therefore, 
\begin{equation}\label{aff30v}
L_k^W = \sum_{i=1}^{\ell} \sum_{n\in\ZZ} (v_i t^{\al_0+n+k}) (v_{2\ell+1-i} t^{-\al_0-n}) \,, \qquad k\ne 0\,.
\end{equation}

The triangular decomposition of $\hhp$ depends on whether $\al_0\in\CC^-$ or $\al_0=0$. Suppose first that $\al_0\in\CC^-$. Then $\hhp^0=\CC K$ and $R=\CC$ with $K=I$ acting as the identity operator. We have:
\begin{equation}\label{aff20}
M_\ph(R) \cong \CC[x_{i,0}, x_{j,n}]_{1\le i\le\ell, \,1\le j\le 2\ell, \, n=1,2,3,\dots} \,,
\end{equation}
where for $1\le i\le\ell$, $n=0,1,2,\dots,$
\begin{equation*}%\label{aff20}
v_i t^{\al_0-n} = x_{i,n} \,, \qquad v_{\ell+i} t^{-\al_0-n-1} = x_{\ell+i,n+1}  \,,
\end{equation*}
and
\begin{align*}%\label{aff21}
v_i t^{\al_0+n+1} &= (\al_0+n+1) \d_{x_{2\ell+1-i,n+1}} +(1-\de_{i,\ell}) \d_{x_{2\ell-i,n+1}} \,, \\ 
v_{\ell+i} t^{-\al_0+n} &= (-\al_0+n) \d_{x_{\ell+1-i,n}} -(1-\de_{i,\ell}) \d_{x_{\ell-i,n}} \,.
\end{align*}

\begin{lemma}\label{laff6}
For\/ $\al_0\in\CC^-$ and\/ $W=M_\ph(R)$ as in \eqref{aff20}, we have
\begin{align*}%\label{aff30}
L_0^W &= \sum_{i=1}^\ell \sum_{n=0}^\infty x_{i,n} \bigl( (-\al_0+n) \d_{x_{i,n}} -(1-\de_{i,1}) \d_{x_{i-1,n}} \bigr) \\
&+ \sum_{i=1}^\ell \sum_{n=1}^\infty x_{\ell+i,n} \bigl( (\al_0+n) \d_{x_{\ell+i,n}} +(1-\de_{i,1}) \d_{x_{\ell+i-1,n}} \bigr) \\
&- \frac\ell{2}(\al_0^2+\al_0) I \,.
\end{align*}
\end{lemma}
\begin{proof}
Let us apply \eqref{aff13o} with $v^i=v_{2\ell+1-i}$. We observe that there are only two cases in which the normally ordered product from \deref{dlogf4} differs from the one obtained by placing all $x_{i,n}$ to the left of all $\d_{x_{i,n}}$. 
First, for $\ell+1\le i\le 2\ell$,
\begin{align*}
{:} (v^i t^{\al_0}) & (v_i  t^{-\al_0}) {:} = (v_i  t^{-\al_0}) (v^i t^{\al_0}) \\
&= x_{2\ell+1-i,0} \bigl( (-\al_0) \d_{x_{2\ell+1-i,0}} -(1-\de_{i,2\ell}) \d_{x_{2\ell-i,0}} \bigr) -\al_0 I \,.
\end{align*}
Second, for $\Re\al_0<0$ and $1\le i\le \ell$,
\begin{align*}
{:} (v^i t^{-\al_0-1}) & (v_i  t^{\al_0+1}) {:} = (v_i  t^{\al_0+1}) (v^i t^{-\al_0-1}) \\
&= x_{2\ell+1-i,1} \bigl( (\al_0+1) \d_{x_{2\ell+1-i,1}} +(1-\de_{i,\ell}) \d_{x_{2\ell-i,1}} \bigr) + (\al_0+1) I \,.
\end{align*}

On the other hand, when $\Re\al_0<0$, we have
\begin{equation*}
\S v_i = \al_0 v_i \,, \quad \S v_{\ell+i} = (-\al_0-1) v_{\ell+i} \,, \qquad 1\le i\le \ell \,,
\end{equation*}
from where we find $\tr\binom{\S}2 = \ell (\al_0^2+\al_0+1)$. 
When $\Re\al_0=0$, we have
 \begin{equation*}
\S v_i = \al_0 v_i \,, \quad \S v_{\ell+i} = -\al_0 v_{\ell+i} \,, \qquad 1\le i\le \ell \,,
\end{equation*}
which gives $\tr\binom{\S}2 = \ell \al_0^2$. In both cases, the combined contribution from the difference of the normally ordered products and $\tr\binom{\S}2$ is $- \frac\ell{2}(\al_0^2+\al_0)$.
\end{proof}

Now consider the case when $\al_0=0$ in \exref{eaff5}. Then $\hhp^0=\lieh t^0\oplus\CC K$ is a direct sum of a finite-dimensional Heisenberg Lie algebra and the central ideal $\Span\{v_\ell t^0,v_{2\ell} t^0\}$. We have the following irreducible $\hhp^0$-modules $R$ with $K=I$:
\begin{equation*}%\label{aff22}
R_{a_1,a_2} = \CC[x_{i,0}]_{1\le i\le\ell-1} \qquad (a_1,a_2 \in \CC) \,,
\end{equation*}
where for $1\le i\le\ell-1$,
\begin{equation*}%\label{aff22}
v_i t^0 = x_{i,0} \,, \quad v_{\ell+i} t^0 = -\d_{x_{\ell-i,0}}  \,, \quad v_\ell t^0=a_1 I \,, \quad v_{2\ell} t^0=a_2 I \,.
\end{equation*}
Then
\begin{equation}\label{aff23}
M_\ph(R_{a_1,a_2}) \cong \CC[x_{i,0}, x_{j,n}]_{1\le i\le\ell-1, \,1\le j\le 2\ell, \, n=1,2,3,\dots} \,,
\end{equation}
where for $1\le i\le\ell$ and $n=1,2,3,\dots,$
\begin{align*}%\label{aff24}
v_i t^{-n} &= x_{i,n} \,, \qquad v_{\ell+i} t^{-n} = x_{\ell+i,n}  \,, \\
v_i t^{n} &= n \d_{x_{2\ell+1-i,n}} +(1-\de_{i,\ell}) \d_{x_{2\ell-i,n}} \,, \\ 
v_{\ell+i} t^{n} &= n \d_{x_{\ell+1-i,n}} -(1-\de_{i,\ell}) \d_{x_{\ell-i,n}} \,.
\end{align*}

We can determine $L_0^W$ as in \leref{laff6}; however, now there is no problem with the normally ordered products because $v^i t^0$ commutes with $v_i t^0$. We obtain
for $\al_0=0$ and $W=M_\ph(R_{a_1,a_2})$:
\begin{align*}%\label{aff31}
L_0^W &= \sum_{i=1}^\ell \sum_{n=1}^\infty x_{i,n} \bigl( n \d_{x_{i,n}} -(1-\de_{i,1}) \d_{x_{i-1,n}} \bigr) \\
&+ \sum_{i=1}^\ell \sum_{n=1}^\infty x_{\ell+i,n} \bigl( n \d_{x_{\ell+i,n}} +(1-\de_{i,1}) \d_{x_{\ell+i-1,n}} \bigr) \\
&- \sum_{i=2}^{\ell-1}  x_{i,0} \d_{x_{i-1,0}} + a_2 x_{1,0} - a_1 \d_{x_{\ell-1,0}} \,.
\end{align*}

\subsection{The case $\dim\lieh=2\ell-1$}\label{stwheis2}

Now let $\lieh$ be as in \exref{eaff6}. Then $\hat\lieh_\ph$ is the Lie algebra spanned by a central element $K$ and elements $v_i t^{\al_0+n}$ ($1\le i\le 2\ell-1$, $n\in\ZZ$), with Lie brackets given by:
\begin{equation*}%\label{aff26}
[v_i t^m, v_j t^k] = m \de_{m+k,0}\de_{i+j,2\ell} K + (-1)^{i+1} \de_{m+k,0} \de_{i+j,2\ell-1} K \,.
\end{equation*}
It follows from \eqref{twlog10} that for $1\le j\le 2\ell-1$:
\begin{equation*}%\label{aff26y}
Y(v_j,\z) = \sum_{i=j}^{2\ell-1} \sum_{m\in\al_0+\ZZ} (-1)^{(i-j)(i+j-1)/2} \, \ze^{(i-j)} (v_i t^m) \z^{-m-1} \,.
\end{equation*}
The action of the Virasoro algebra is again determined by \eqref{aff13o} with $\bar\om=0$.
The dual basis $\{v^i\}$ to the basis $\{v_i\}$ is given by $v^i=v_{2\ell-i}$; therefore
\begin{equation}\label{aff32v}
L_k^W = \frac12 \sum_{i=1}^{2\ell-1} \sum_{n\in\ZZ} (v_i t^{\al_0+n+k}) (v_{2\ell-i} t^{-\al_0-n}) \,, \qquad k\ne 0\,.
\end{equation}

Suppose first that $\al_0=-1/2$; then $\hhp^0=\CC K$ and $R=\CC$ with $K=I$.
We have:
\begin{equation}\label{aff27}
M_\ph(R) \cong \CC[x_{j,n}]_{1\le j\le 2\ell-1, \, n=0,1,2,\dots} \,,
\end{equation}
where for $1\le i\le 2\ell-1$ and $n=0,1,2,\dots,$
\begin{align*}%\label{aff28}
v_i t^{-\frac12-n} &= x_{i,n} \,,\\ 
v_i t^{\frac12+n} &= \Bigl(\frac12+n\Bigr) \d_{x_{2\ell-i,n}} + (-1)^{i+1} (1-\de_{i,2\ell-1}) \d_{x_{2\ell-1-i,n}}  \,.
\end{align*}
As in \leref{laff6}, we find that for $W=M_\ph(R)$,
\begin{align*}%\label{aff32}
L_0^W &= \sum_{i=1}^{2\ell-1} \sum_{n=0}^\infty 
x_{i,n} \Bigl( \Bigl(\frac12+n\Bigr) \d_{x_{i,n}} + (-1)^{i+1} (1-\de_{i,1}) \d_{x_{i-1,n}} \Bigr) \\
&+  \frac1{16} (2\ell-1) I \,.
\end{align*}

Now let $\al_0=0$. Then $\hhp^0=\lieh t^0\oplus\CC K$ is a direct sum of a finite-dimensional Heisenberg Lie algebra and the central ideal $\Span\{v_{2\ell-1} t^0\}$.
The following are irreducible $\hhp^0$-modules with $K=I$:
\begin{equation*}%\label{aff29}
R_a = \CC[x_{i,0}]_{1\le i\le\ell-1} \qquad (a\in \CC) \,,
\end{equation*}
where
\begin{equation*}%\label{aff24}
v_i t^0 = x_{i,0} \,, \quad v_{\ell-1+i} t^0 = (-1)^{\ell-i} \d_{x_{\ell-i,0}}  \,, \quad v_{2\ell-1} t^0=a I \,,
\end{equation*}
for $1\le i\le\ell-1$. Then
\begin{equation}\label{aff25}
M_\ph(R_a) \cong \CC[x_{i,0}, x_{j,n}]_{1\le i\le\ell-1, \, 1\le j\le 2\ell-1, \, n=1,2,3,\dots} \,,
\end{equation}
where for $1\le i\le 2\ell-1$ and $n=1,2,3,\dots,$
\begin{align*}%\label{aff26}
v_i t^{-n} &= x_{i,n} \,,\\ 
v_i t^{n} &= n \d_{x_{2\ell-i,n}} + (-1)^{i+1} (1-\de_{i,2\ell-1}) \d_{x_{2\ell-1-i,n}} \,.
\end{align*}
For $W=M_\ph(R_a)$, we have
\begin{align*}%\label{aff33}
L_0^W &= \sum_{i=1}^{2\ell-1} \sum_{n=1}^\infty 
x_{i,n} \bigl( n \d_{x_{i,n}} + (-1)^{i+1} (1-\de_{i,1}) \d_{x_{i-1,n}} \bigr) \\
&+ \sum_{i=2}^{\ell-1}  (-1)^{i+1} x_{i,0} \d_{x_{i-1,0}} + \frac12 \d_{x_{\ell-1,0}}^2 + a x_{1,0}  \,.
\end{align*}

\begin{remark}\label{raff}
After a change of variables, the Virasoro operators \eqref{aff30v}, \eqref{aff32v} with $\al_0=a_1=a_2=a=0$ coincide with those of 
\cite{EHX, EJX} (see also \cite{DZ1, DZ2}). The detailed correspondence will be discussed elsewhere.
\end{remark}

The following special case of \exref{eaff6} is related to \cite{M}.

\begin{example}\label{eaff7}
Consider an affine Kac--Moody algebra of type $A_1^{(1)}$, and let $\lieh$ be its Cartan subalgebra. The dual space $\lieh^*$ has a basis $\{\al_1,\de,\La_0\}$ and a nondegenerate symmetric bilinear form $(\cdot|\cdot)$ given by:
\begin{equation*}%\label{aff34}
(\al_1|\al_1)=2 \,, \qquad (\de|\La_0)=(\La_0|\de)=1 \,,
\end{equation*}
where all other products of basis vectors are $0$ (see \cite[Chapter 6]{K1}). The affine Weyl group has an element $\ph=t_{\al_1}$, which acts on $\lieh^*$ by:
\begin{equation*}%\label{aff35}
\ph(\al_1)=\al_1-2\de \,, \qquad \ph(\de)=\de  \,, \qquad \ph(\La_0)=\La_0+\al_1-\de \,.
\end{equation*}
The bilinear form $(\cdot|\cdot)$ is $\ph$-invariant. Introduce another basis
\begin{equation*}%\label{aff36}
v_1=-\frac{2\pi\ii}{\sqrt2} \, \La_0 \,, \qquad v_2=\frac{\al_1}{\sqrt2}  \,, \qquad v_3=-\frac{\sqrt2}{2\pi\ii} \, \de \,,
\end{equation*}
so that $(v_i|v_j)=\de_{i+j,4}$. Then $\ph=e^{-2\pi\ii\N}$ where $\N$ is the linear operator defined by
$\N(v_1)=v_2$, $\N(v_2)=-v_3$ and $\N(v_3)=0$.
\end{example}

%\section{Conclusions}\label{s7}

\section*{Acknowledgements}
This paper was motivated by my joint work \cite{BM} with Todor Milanov and our ongoing collaboration.  I would like to thank him for many stimulating discussions. I am grateful to Di Yang for sharing his unpublished manuscript \cite{LYZ}, and to Dra\v{z}en Adamovi\'c and Antun Milas for discussions on logarithmic CFT. This research was supported in part by a Simons Foundation grant.

%%%%%%%%%%%%%%%%%%%%%%%%%%%%%%%%%%%
\bibliographystyle{amsalpha}

\end{document}